\documentclass[11pt]{article} 

\usepackage[utf8]{inputenc} 

\usepackage{authblk}

\usepackage{geometry} 
\usepackage{graphicx} 

\usepackage{float}			
\usepackage{comment}
\usepackage[english]{babel}
\usepackage{graphicx} 
\usepackage{caption}
\usepackage{subcaption}
\usepackage{verbatim} 
\usepackage{amsmath}
\usepackage{amssymb}
\usepackage{amsfonts} 

\usepackage{amsthm}

\theoremstyle{plain}

\newtheorem{thm}{Theorem}

\newtheorem{lem}[thm]{Lemma}
\newtheorem*{lem*}{Lemma}
\newtheorem{prop}[thm]{Proposition}

\newtheorem{cor}[thm]{Corollary}
\theoremstyle{definition}

\newtheorem{defn}[thm]{Definition}

\theoremstyle{remark}

\newtheorem*{rem}{Remark}

\theoremstyle{plain}
\newtheorem*{hyp}{}

\newcommand\RR{\mathbb{R}} 

\title{Quantitative convergence towards a self similar profile in an age-structured renewal equation for subdiffusion.}

\author[1]{Berry, Hugues \thanks{hugues.berry@inria.fr}}
\author[1,2]{Lepoutre, Thomas\thanks{thomas.lepoutre@inria.fr}}
\author[1,3]{Mateos González, Álvaro \thanks{alvaro.mateos\textunderscore gonzalez@ens-lyon.fr}}
\affil[1]{Inria, 56 Blvd Niels Bohr, F-69603 Villeurbanne, France}
\affil[2]{Universit\'e de Lyon, Institut Camille Jordan,
CNRS UMR 5208,
Universit\'e Claude Bernard Lyon 1,
43 blvd. du 11 novembre 1918
F-69622 Villeurbanne cedex
France}
\affil[3]{Universit\'e de Lyon, Unité de Mathématiques Pures et Appliquées, CNRS UMR 5669, École Normale Supérieure de Lyon, 15 parvis René Descartes, F-69007 Lyon, France}

\begin{document}

\maketitle

\begin{abstract}
Continuous-time random walks are generalisations of random walks frequently used to account for the consistent observations that many molecules in living cells undergo anomalous diffusion, i.e. subdiffusion. Here, we describe the subdiffusive continuous-time random walk using age-structured partial differential equations with age renewal upon each walker jump, where the age of a walker is the time elapsed since its last jump. In the spatially-homogeneous (zero-dimensional) case, we follow the evolution in time of the age distribution. An approach inspired by relative entropy techniques allows us to obtain quantitative explicit rates for the convergence of the age distribution to a self-similar profile, which corresponds to convergence to a stationnary profile for the rescaled variables. An important difficulty arises from the fact that the equation in self-similar variables is not autonomous and we do not have a specific analyitcal solution. Therefore, in order to quantify the latter convergence, we estimate attraction to a time-dependent ``pseudo-equilibrium'', which in turn converges to the stationnary profile.
\end{abstract}

\vspace{0.5cm}

\small{\noindent
}

\vspace{0.5cm}


\section{Introduction}
\label{sec_intro}
\subsection{Brief model description}

Recent methodological advances in cell biology allowed the measurements of the displacement of single molecules (or assemblies thereof) in living single cells. Those investigations have consistently reported that the random displacement of molecules inside cells often deviates from Brownian motion, with the mean squared displacement $\left\langle \mathbf{r}^2(t)\right \rangle$ that does not scale linearly with time, as in Brownian motion, but sublinearly, with a power-law behavior: $\left\langle \mathbf{r}^2(t)\right \rangle \propto t^\mu$~\cite{Cox,Bronstein,Parry2014,DiRienzo2014}. This behavior is usually referred to as ``anomalous'' diffusion or ``subdiffusion'', since $\mu<1$ usually, for non-active transport (for a review see e.g. \cite{Hoefling-Franosch}).

Continuous-time random walks (CTRW) are one of the main mechanisms that are recurrently evoked to explain the emergence of subdiffusion in cells. CTRW were introduced fifty years ago by Montroll and Weiss as a generalisation of random walks \cite{Montroll1965}, where the residence time (the time between two consecutive jumps) is a random variable $\tau$ with probability distribution $\phi(\tau)$ (see \cite{Metzler2000} for a review). If the expectation of $\tau$ is defined, for instance when $\tau$  is dirac-distributed or decays exponentially fast, one recovers the ``normal'' Brownian motion. However, when the expectation of $\tau$ diverges, for instance when $\phi(\tau)$ is heavy-tailed, $\phi(\tau)\propto \tau^{-(1+\mu)}$ with $0<\mu<1$, the CTRW describes a subdiffusive behavior, with $\left\langle \mathbf{r}^2(t)\right \rangle \propto t^\mu$.

One great achievement of CTRW is that they can readily be used to derive mean-field equations for the spatio-temporal dynamics of the random walkers. Indeed, starting from $\phi(\tau)$, combinations of Laplace and Fourier transforms lead to a ``subdiffusion'' equation for the density of random walkers located at position $ \mathbf{x}$ at time $t$: $\partial_t\rho(\mathbf{x},t)=D_\mu \mathcal{D}_t^{1-\mu}\nabla^2\rho(\mathbf{x},t) $ where $D_\mu$ is a generalised diffusion coefficient and $\mathcal{D}_t^{1-\mu}$ is the Riemann-Liouville fractional derivative operator \cite{Metzler2000,Reaction-Transport}. Such a fractional dynamics formulation is very attractive for modelling in biology, in particular because of its apparent similarity with the classical diffusion equation. However, contrarily to the diffusion equation, the Rieman-Liouville operator is non-Markovian. This non-Markovian property becomes a serious obstacle when one wants to couple subdiffusion with chemical reaction \cite{Henry2006,Yuste2008,Fedotov2014}.  

Here, we take an alternative approach to CTRW that maintains the Markovian property of the transport equation at the price of a supplementary independent variable. We associate each random walker with an age $a$, that is reset when the random walker jumps. In one dimension of space, we note $n(t,x,a)$ the density probability distribution of walkers at time $t$ that have been residing at location $x$ during the last span of time $a$. The dynamics of the CTRW is then described with an age-renewal equation with spatial jumps that reinitialise the age:
\begin{equation}
\label{eqn_n_space}
\left\{ \begin{array}{l} \partial_t n(t,x,a) + \partial_a n(t,x,a) + \beta(a) n(t,x,a) = 0\, ,\quad t \geq0, \quad a > 0\, , \quad x\in \mathbb{R}\smallskip\\
n(t,x,a=0) = \int_0^\infty \int_\RR \beta(a') \omega(x-x') n(t,x',a') {\rm d}x {\rm d}a' \smallskip\\
n(t = 0,x,a) = n^0(x,a).
\end{array}\right.
\end{equation}

The kernel $\omega$ describes the spatial distribution of jump destinations (typically a Gaussian distribution centred at the origin position), and the function $\beta(a)$ gives the jump rate. Since we are mostly interested here in the subdiffusive case (where the expectation of the residence time diverges), we will focus throughout this article on the case:
\begin{equation}
a \beta(a) \xrightarrow[a \to \infty]{} \mu \in (0,1).
\end{equation}
The precise meaning of the limit will be given later on. The limit $\mu$ in eq.\eqref{eqn_n_space} is the subdiffusion exponent: for $\mu > 1$, eq.\eqref{eqn_n_space} describes a diffusive process, whereas for $0 < \mu < 1$ the mean time a particle has to wait between two consecutive renewals diverges and the mean squared displacement exhibits subdiffusion with exponent $\mu$. The distribution of residence time $\phi(\tau)$ evoked above is related to the jump rate as: $\phi(\tau) = \beta(\tau) \exp \left( -\int_0^\tau  \beta(s) \, {\rm d}s \right)$. Note that this age-structured approach is not uncommon in the CTRW literature \cite{Reaction-Transport,subdiff}. Our main contribution here is to use it in conjunction with approaches borrowed from the study of partial differential equations.

In the present article, we restrict our attention to the temporal evolution of the age distribution of the walkers. To this end, we simplify the problem by considering its spatially-homogenous version, namely:

\begin{equation}
\label{eqn_n}
\left\{ \begin{array}{l} \partial_t n(t,a) + \partial_a n(t,a) + \beta(a) n(t,a) = 0\, ,\quad t \geq0, \quad a > 0 \smallskip\\
n(t,a=0) = \int_0^\infty \beta(a') n(t,a') {\rm d}a' \smallskip\\
n(t = 0,a) = n^0(a).
\end{array}\right.
\end{equation}

\subsection{Self-similar solutions}

The only steady state solution of eq.\eqref{eqn_n} in $L^1$ is $0$, which doesn't allow us to describe the dynamics of the system in a satisfactory way. Hence the search for self-similar solutions. An educated guess is that they should be of the following form, with $A(t)$ to be determined:
\[
n(t,a) = \frac{1}{A(t)}w \Big(\underbrace{\ln(1+t)}_{\tau}, \underbrace{a / A(t)}_{b}\Big) .
\]

Let us consider, for the sake of simplicity, an initial condition supported on $[0,1)$. By injecting the previous expression into eq.\eqref{eqn_n}, we find that the natural choice $A(t) = 1+t$ preserves the initial condition $n(0,b)\text{:}=n^0(b)= w(0 ,b) \text{:}= w^0(b)$, and yields :

\begin{equation}
\label{eqn_w}
\left\{ \begin{array}{l} \partial_\tau w + \partial_b ((1-b)w) + {\rm e}^\tau \beta({\rm e}^\tau b) w = 0\smallskip\\
w(\tau, 0) = \int_0^\infty  {\rm e}^\tau \beta({\rm e}^\tau b) w(\tau, b) {\rm d}b \smallskip\\
w(0,b) = w^0(b).
\end{array}\right.
\end{equation}

Note that for an initial condition supported on $[0,A_+)$, $A(t) = A_+ + t$ is a better choice and leads to a similar analysis. If the initial condition is not compactly supported, the tail of the age distribution can influence the convergence rate we give below.

It is important to note that the previous system is not autonomous, for the term ${\rm e}^\tau \beta({\rm e}^\tau b)$ depends on $b$. This rescaling does not lead to a classical steady state, and we could not find a particular solution of the previous equation. However, we may look for a stationary state satisfying formally the following equation, since we consider here $\beta(a) \sim \frac \mu a$.
\[
 \partial_b ((1-b)W_\infty) + \frac \mu b W_\infty = 0\smallskip\\
\]
where the boundary condition cannot be stated as an equality since $W_\infty$ is expected to blow up at $0$, but can be understood as an equivalence as $\varepsilon$ tends to $0$ of $W_\infty (\varepsilon)$ and $\int_\varepsilon^\infty \frac \mu b W_\infty(b) {\rm d}b$.

This leads us to define the self-similar equilibrium as:
\begin{equation}\label{eqWinf}
W_\infty (b) = \frac{c_\infty}{b^\mu (1-b)^{1-\mu}}
\end{equation}
which is called the arcsine distribution, or Dynkin-Lamperti distribution. $c_\infty$ is defined such that $\| W_\infty \|_1 = 1$. Under some conditions, we can expect that $w(\tau, b)$ will converge to eq.\eqref{eqWinf} when $\tau \to \infty$.

A similar result in probability theory appears in Feller's book \cite{book_Feller} tome II, chapter XI, especially in section 5 and onwards, where the renewal problem is tackled by considering the waiting time before the $n^{\rm th}$ renewal. For an introduction to renewal theory, see the eponymous chapter (8.6) in \cite{book_Regular_Variation}. However, no convergence rate is given for our infinite mean waiting time problem in any of these books, and we have been unable to locate such a convergence rate in the subsequent literature. Recent developments in Ergodic Theory for mildly related problems (see chapter 8.11 of \cite{book_Regular_Variation} for an introduction to Darling-Kac theory), have yielded convergence rates, that are optimal in certain cases, as shown in \cite{Inventiones} and \cite{DK_error}. 


\subsection{Main results}

Throughout the article, the following set of hypotheses will intervene. Hypothesis \textbf{(H1)} will be used in properties of convergence without a rate while hypothesis \textbf{(H2)} will allow convergence rate estimates.

\begin{hyp}[\textbf{H1}]
$\beta$ is a positive, bounded, and non-increasing function satisfying 
$$\lim_{a \to \infty} a \beta(a) = \mu \in (0,1).$$
\end{hyp}

\begin{rem}
We will always assume $\beta$ to be non-increasing for the sake of simplicity (in particular in theorem \ref{thm_beta_gral}, propositions \ref{prop_H_to_0} and \ref{prop_conv_g}, and lemmas \ref{lem_H'_leq} and \ref{lem_W_W_infty}). The monotonicity can be replaced by the following hypothesis: 

$\beta$ is a positive, bounded function satisfying $\lim_{a \to \infty} a \beta(a) = \mu \in (0,1)$, such that $\underline \beta$ defined as follows
\begin{equation}\label{eq_beta_bar}
\underline \beta(a) = \inf_{x\in[0,a]} \beta (x)
\end{equation}
also satisfies $\lim_{a \to \infty} a \underline \beta(a) = \mu$.

This leads to minor changes in the proofs, the loss of a multiplicative constant in the affected results and replacing $\beta$ by $\underline \beta$ where it corresponds.
\end{rem}

\begin{hyp}[\textbf{H2}]
$\beta$ satisfies \textbf{(H1)}. Additionally, $\beta(a) = \frac \mu {1+a} + g(a)$, where $g \in L^1$ and there exist $K, \alpha > 0$ such that
\[
\int_a^\infty |g(s)| {\rm d} s \leq \frac K {(1+a)^\alpha}.
\]
\end{hyp}

\begin{rem}
For the sake of clarity we will investigate separately the particular case $g = 0$, called the "reference case". Then, all our results will be extended to the general case at the expense of the convergence rates.
\end{rem}

Due to the specific shape of $W_\infty$ and to the boundary condition, it is difficult to investigate in a direct way the evolution of $\| w - W_\infty \|_1$: the methods we describe subsequently fail to do so. However, we could recover a quantitative explicit convergence rate with respect to a ``pseudo-equilibrium'' $W$ which will be proved to converge in $L^1$ to $W_\infty$.

\begin{defn}
\label{def_W_tau}
We define the pseudo-equilibrium $W$ over $\RR_+ \times [0,1)$ as follows :
\begin{equation}
W(\tau, b) = \frac{C(\tau)}{{\rm e}^{B({\rm e}^\tau b)}(1-b)^{1-\mu} }. 
\end{equation}
where $B(a) = \int_0^a \beta(s) {\rm d} s$ and $C$ is defined so that $\| W(\tau, \cdot) \|_{L^1} = 1$.
\end{defn}
In particular in the reference case $\beta(a)=\frac{\mu}{1+a}$, it may be written as 
$$
W(\tau,b)=\frac{C(\tau){\rm e}^{-\mu\tau}}{({\rm e}^{-\tau}+b)^\mu(1-b)^{1-\mu}}.
$$

Note the similarity between this expression and that for $W_\infty$ in eq.\eqref{eqWinf}. In the following, we obtain explicit convergence rates of $w(\tau,b)$ to $W$, developing proofs based on Relative Entropy estimates. Importantly, we show that $W$ converges to $W_\infty$ at the same rate (up to multiplication by a constant) as the rate with which $w$ converges to $W$. Hence, the convergence to the pseudo-equilibrium $W$ yields a very good estimate of the convergence to the self-similar equilibrium $W_\infty$. Finally, we carry out Monte-Carlo simulations of zero-dimensional CTRW to illustrate and question the optimality of our main analytical results.

\begin{defn}
For the moment and for the sake of simplicity, let us define:
\[
\mathcal H (\tau) = \| w(\tau, \cdot) - W(\tau, \cdot) \|_{L^1([0,1])}.
\]
\end{defn}

\begin{rem}
Later on, we shall define more generally $\mathcal H$ as a relative entropy, the $L^1$ distance being a particular case more suited to our purposes.
\end{rem}

In this paper, we prove the following propositions:

\begin{prop}
\label{prop_H_to_0}
Under hypothesis \textbf{(H1)}, we have:
\[
\mathcal H(\tau) \xrightarrow[\tau \to \infty]{} 0.
\]
\end{prop}

Our first quantitative result is a convergence rate for the reference case of hypothesis \textbf{(H2)}.

\begin{thm}
\label{thm_beta_ref}
Let $\beta(a) = \frac \mu {1+a}$. Then we have the following convergence rates:
\begin{equation}
\left\{
\begin{array}{ll}
\mathcal H (\tau) \leq {\rm e}^{-\mu \tau} \left[ \mathcal H (0) \left( \frac 2 {1 + {\rm e}^{-\tau}} \right)^\mu - \frac 8 {2\mu - 1} \left( \frac 2 {1 + {\rm e}^{-\tau}}\right)^\mu \right] + {\rm e}^{-(1-\mu) \tau} \left[  \frac 8 {2\mu - 1} \left( \frac 2 {1 + {\rm e}^{-\tau}}\right)^\mu  \right] & {\text \quad if \quad} \mu \neq \frac 1 2 \medskip\\
\mathcal H (\tau) \leq {\rm e}^{-\frac \tau 2} \left[ \frac 2 {1 + {\rm e}^{-\tau}} \left( \mathcal H(0) + 8 \tau \right) \right] & {\text \quad if \quad} \mu = \frac 1 2.
\end{array}
\right.
\end{equation}
\end{thm}

A modified, yet analogous, convergence rate still holds for $g \neq 0$:
\begin{thm}
\label{thm_beta_gral}
Suppose hypothesis \textbf{(H2)} holds.\\
If $\alpha >1-\mu$, we recover the optimal rate of convergence
$$
\mathcal{H}(\tau)\leq \begin{cases}K({\rm e}^{-\mu\tau}+{\rm e}^{-(1-\mu)\tau}) , \quad & \text{ if }\mu\not=1/2 \\
K\tau {\rm e}^{-\tau/2}, \quad & \text{ if } \mu=1/2.
\end{cases}
$$
If $\alpha\leq 1-\mu $, we need to distinguish between several cases:
$$
\mathcal{H}(\tau)\leq \begin{cases}
K({\rm e}^{-\alpha\tau}+{\rm e}^{-\mu\tau}), \quad & \text{ if } \mu\not=\alpha <1-\mu\\[0.3cm]
K(1+\tau){\rm e}^{-\mu\tau}, \quad & \text{ if }\alpha=\mu<1-\mu,\\[0.3cm]
K(\tau {\rm e}^{-(1-\mu)\tau}+{\rm e}^{-\mu\tau}), \quad & \text{ if }\alpha=1-\mu\not=1/2,\\[0.3cm]
K(1+\tau^2){\rm e}^{-\tau/2}, \quad & \text{ if }\alpha=\mu=1-\mu=1/2.
\end{cases}
$$
\end{thm}

We finally reinterpret our results in terms of non-rescaled variables, for instance in the reference case $\beta(a)=\frac{\mu}{1+a}$.
\begin{cor}
\label{cor_N}
Assume $n^0$ is supported in $[0,1)$ and $\beta(a)=\frac{\mu}{1+a}$, then if we denote 

$$
N_\infty(t,a)=
\begin{cases}
\frac{c_\infty}{a^\mu(1+t-a)^{1-\mu}}, & a<1+t,\\
0, & a>1+t. 
\end{cases}
$$
Then if $\mu\not=1/2$, there exists $K$ such that 
$$
\|n(t,.)-N_\infty(t,.)\|_1\leq \frac{K}{(1+t)^{\mu}}+\frac{K}{(1+t)^{1-\mu}}.
$$
If $\mu=1/2$, then we have 
$$
\|n(t,.)-N_\infty(t,.)\|_1\leq \frac{K(1+\log(1+t))}{\sqrt{1+t}}.
$$
\end{cor}

\begin{rem}
An analogous version in non-rescaled variables can be given for theorem \ref{thm_beta_gral}.
\end{rem}

\subsection{Outline of the paper}
The paper is organized as follows.
In Section~\ref{sec_ES} we set the entropic structure of the equation and the main properties of the pseudo equilibrium $W$. In particular we establish (non quantitatively) that $$\lim_{\tau \to \infty} \|w(\tau, \cdot)-W(\tau,\cdot)\|_1 = 0,\qquad \lim_{\tau \to \infty} \|W_\infty(\tau, \cdot)-W(\tau, \cdot)\|_1 = 0,$$ proving thereby  
$$
\lim_{\tau \to \infty} \|w(\tau, \cdot)-W_\infty(\tau, \cdot)\|_1 = 0.
$$ 
Section \ref{sec_beta} deals with quantitative convergence rates towards the pseudo-equilibrium $W$, proving Theorems \ref{thm_beta_ref} and \ref{thm_beta_gral}. A convergence rate for $\| w - W_\infty \|_1$, some effects of initial conditions on the convergence rates, and convergence rates in non-rescaled variables are dealt with in Section \ref{sec_other_CR}. Finally, we show the results of some simulations in the eponymous section.


\section{Entropic structure}\label{sec_ES}
Even if we are mainly estimating  $L^1-$norms, we see our proof as a specific case of relative entropy inequalities. Rates could be obtained following the lines of our proofs for other entropies. 
\subsection{$L^1$ contraction for compactly supported solutions}\label{ssec_2sol}
The first evidence of an attractor is the $L^1$ contraction of compactly supported solutions. If we consider two initial data supported in $[0,1)$, $w_1^0,w_2^0$, non-negative and of mass 1 and the associated solutions $w_1,w_2$, then we have the following property (we take $\beta(a)=\frac{\mu}{1+a}$ for this computation)

$$
\frac{d}{d\tau}\int_0^1 |w_1(\tau,b)-w_2(\tau,b)|db \leq D(\tau)
$$
where
$$
D(\tau) = \left|\int_0^1\frac{\mu}{{\rm e}^{-\tau}+b}(w_1(\tau,b)-w_2(\tau,b))db\right|-\int_0^1\frac{\mu}{{\rm e}^{-\tau}+b}\left|w_1(\tau,b)-w_2(\tau,b)\right|db.
$$ 
Since mass is conserved i.e., $\int_0^1 w_1-w_2=0$, we have easily
\begin{align*}
D(\tau)=\left|\int_0^1\left(\frac{\mu}{{\rm e}^{-\tau}+b}-\frac{\mu}{{\rm e}^{-\tau}+1}\right)(w_1(\tau,b)-w_2(\tau,b))db\right|\\-\int_0^1\frac{\mu}{{\rm e}^{-\tau}+b}\left|w_1(\tau,b)-w_2(\tau,b)\right|db
\\
\leq \int_0^1\left(\frac{\mu}{{\rm e}^{-\tau}+b}-\frac{\mu}{{\rm e}^{-\tau}+1}\right)|w_1(\tau,b)-w_2(\tau,b)|db\\-\int_0^1\frac{\mu}{{\rm e}^{-\tau}+b}\left|w_1(\tau,b)-w_2(\tau,b)\right|db.
\end{align*}
Thereby, we obtain
$$
\frac{d}{d\tau}\int_0^1 |w_1(\tau,b)-w_2(\tau,b)|db\leq -\frac{\mu}{1+{\rm e}^{-\tau}}\int_0^1 |w_1(\tau,b)-w_2(\tau,b)|db.
$$
And this leads to 
$$
\int_0^1 |w_1(\tau,b)-w_2(\tau,b)|db=O({\rm e}^{-\mu\tau}).
$$
In the next section we identify the attractor towards which solutions converge.
\subsection{Pseudo equilibrium}

We start by recalling the definition of what we call the pseudo equilibrium.
We recall Definition \ref{def_W_tau}:
\[
W(\tau, b) = \frac{C(\tau)}{(1-b)^{1-\mu} {\rm e}^{B({\rm e}^\tau b)}}
\]
where $B(a) = \int_0^a \beta(s) {\rm d} s$ and $C$ is defined so that $\| W(\tau, \cdot) \|_{L^1} = 1$.
\begin{rem}
By definition, 
\begin{equation}
W( \tau, b=0) = C(\tau).
\end{equation}
\end{rem}
Firstly we establish the fact that $W(\tau,\cdot)$ is an approximation of $W_\infty$
\begin{lem}\label{lemW_to_DL}
Assume hypothesis \textbf{(H1)}. Then, defining the Dynkin-Lamperti distribution as in \eqref{eqWinf}:
$$
W_\infty (b)=\frac{b^{-\mu}(1-b)^{\mu-1}}{\int_0^1 b^{-\mu}(1-b)^{\mu-1}db}
$$
we have 
$$
\lim_{\tau=+\infty}\|W(\tau,\cdot)-W_\infty\|_1=0.
$$
\end{lem}
\begin{proof}
We start with the model case $\beta(a)=\frac{\mu}{1+a}$. In this case, we can write
$$
W(\tau,b)=\frac{({\rm e}^{-\tau}+b)^{-\mu}(1-b)^{\mu-1}}{\int_0^1 ({\rm e}^{-\tau}+b)^{-\mu}(1-b)^{\mu-1}db},
$$
and the result is immediate. 

For the general case, we use the following useful bound on $W$:
\begin{lem}\label{lem_majW}
Under hypothesis \textbf{(H1)}, for any $\eta>0$ satisfying
$\eta<\min(\mu,1-\mu)$, there exists a constant (depending on $\beta$, but not on $\tau$) $C_\eta>0$ such that
$$
W(\tau,b)\leq C_\eta \frac{(1-b)^{\mu-1}({\rm e}^{-\tau}+b)^{-(\mu+\eta)}}{\int_b^1 (1-b')^{\mu-1}({\rm e}^{-\tau}+b')^{-(\mu+\eta)}db'}
$$  
\end{lem}
\begin{proof}
We first notice that there always exists a function $g_\eta\geq 0$, compactly supported, such that 
$$
\beta(a)\leq \frac{\mu+\eta}{1+a}+g_\eta(a).
$$
Thereby, for all $b'\geq b$ and all $\tau\geq 0$, we have 
$$
B({\rm e}^\tau b)-B({\rm e}^\tau b')=-\int_{{\rm e}^\tau b}^{{\rm e}^\tau b'} \beta(s)ds\geq (\mu+\eta)\ln\left(\frac{1+{\rm e}^\tau b}{1+{\rm e}^\tau b'}\right)-\|g_\eta\|_1,
$$
and 
\begin{equation}\label{eBetaub}
{\rm e}^{B({\rm e}^\tau b)-B({\rm e}^\tau b')}\geq {\rm e}^{-\|g_\eta\|_1}({\rm e}^{-\tau}+b')^{-(\mu+\eta)}({\rm e}^{-\tau}+b)^{\mu+\eta} 
\end{equation}
Then we recall that by definition 
$$
W(\tau,b)=\frac{{\rm e}^{-B({\rm e}^\tau b}(1-b)^{\mu-1}}{\int_0^1 {\rm e}^{-B({\rm e}^\tau b')}(1-b')^{\mu-1}db'}\leq \frac{(1-b)^{\mu-1}}{\int_b^1 {\rm e}^{B({\rm e}^\tau b)-B({\rm e}^\tau b')}(1-b')^{\mu-1}db'}.
$$
Inserting \eqref{eBetaub} in the latter, we obtain the result with $C_\eta={\rm e}^{\|g_\eta\|_1}$.
\qed
\end{proof}
It is worth noticing that we can establish with the same proof
$$
\forall \varepsilon\leq 1/2, \int_0^\varepsilon W\leq  C_\eta' \;\varepsilon^{(1-(\mu+\eta))}
$$

We denote $W_{ref}$ for $\beta(a)=\frac{\mu}{1+a}$ and notice that in our general case 
$$
\beta(a)=\frac{\mu}{1+a}+g(a).
$$
We denote $G(a)=\int_0^a g$ and we can write
$$
W(\tau,b)=K(\tau){\rm e}^{G({\rm e}^\tau)-G({\rm e}^\tau b)}W_{ref}(\tau,b).
$$
for some $K(\tau)>0$ that insures the normalisation $\int_0^1 W=1$. 
We introduce some $\eta>0$ as in lemma~\ref{lem_majW}. We already establish in the proof of lemma~\ref{lem_cdelta} that for $\varepsilon\leq 1/2$,
$$
\int_0^{\varepsilon}W\leq C_\eta \varepsilon^{1-(\mu+\eta)}. 
$$
Therefore, 
$$
1\geq K(\tau)\int_{\varepsilon}^1 {\rm e}^{G({\rm e}^\tau)-G({\rm e}^\tau b)}W_{ref}(\tau,b)db \geq  1- C_\eta \varepsilon^{1-(\mu+\eta)}.
$$
Let $\varepsilon>0$ be fixed. 
Furthermore, since $g(a)=o(\frac{1}{a})$, we have 
$$
\sup_{b\geq \varepsilon} |G({\rm e}^\tau)-G({\rm e}^\tau \varepsilon)|=o\left(\int_{{\rm e}^{\tau}\varepsilon}^{{\rm e}^\tau} \frac{da}{a}\right)=o\left(\ln \varepsilon \right)=o(1).
$$
In particular, for $\varepsilon >0$ fixed, 
$$
{\rm e}^{G({\rm e}^\tau)-G({\rm e}^\tau b)} \rightarrow 1,
$$
uniformly on $(\varepsilon,1)$ as $\tau\rightarrow+\infty$. As a consequence, we have for all $\varepsilon>0$
$$
\int_{\varepsilon}^1 {\rm e}^{G({\rm e}^\tau)-G({\rm e}^\tau b)}W_{ref}(\tau,b)db-\int_{\varepsilon}^1 W_{ref}(\tau,b)db\rightarrow 0.
$$
This leads, for any $\varepsilon > 0$, to the bounds:
$$
\limsup_{+\infty} K(\tau)\leq \frac{1}{\int_\varepsilon^1 W_\infty },\quad \liminf_{+\infty} K(\tau)\leq \frac{ 1- C_\eta \varepsilon^{1-(\mu+\eta)}}{\int_\varepsilon^1 W_\infty }.
$$
Letting $\varepsilon \rightarrow 0$, we obtain 
$$
\lim_{+\infty}K(\tau)=1.
$$
What we established proves that, for any $\varepsilon > 0$,
$$
\int_\varepsilon^1 |W-W_{ref}|\rightarrow 0.
$$
We can conclude using lemma~\ref{lem_majW} that we have 
$$
\int_0^1 |W-W_{ref}|\rightarrow 0.
$$
Consequently,
$$
\lim_{+\infty} \|W-W_\infty\|_1=0.
$$
This ends the proof of lemma~\ref{lemW_to_DL}.
\qed
\end{proof}
The main property of the pseudo equilibrium is the following 
\begin{prop}
\label{prop_W_tau}
$W$ satisfies the following system :
\begin{equation}\label{eq_W}
\left\{
\begin{array}{l}
\partial_\tau W(\tau,b) + \partial_b( (1-b) W(\tau,b)) + {\rm e}^\tau \beta ({\rm e}^\tau b) W(\tau,b) = W(\tau, b) C(\tau) \delta(\tau) \smallskip\\
W(\tau, 0) \left( 1 + \delta(\tau) \right) = \int_0^1 {\rm e}^\tau \beta({\rm e}^\tau b) W(\tau,b) {\rm d} b ,
\end{array}
 \right.
\end{equation}
where $\delta(\tau)$ is defined by the equation 
\begin{equation}\label{eq_delta}
\delta(\tau) = \frac{C'(\tau)}{(C(\tau))^2} - \frac \mu {C(\tau)}.
\end{equation}
\end{prop}

\begin{proof}
By computing the partial derivatives of $W$ with respect to $b$ and to $\tau$, we obtain:
\[
\partial_b ((1-b)W(\tau,b)) = - W(\tau,b) \left[\mu + (1-b) {\rm e}^\tau \beta ({\rm e}^\tau b) \right]
\]
and:
\[
\partial_\tau W(\tau,b) = W(\tau,b) \left[ \frac{C'(\tau)}{C(\tau)} - b {\rm e}^\tau \beta({\rm e}^\tau b) \right].
\]

Therefore, $W$ satisfies :
\[
\partial_\tau W(\tau,b) + \partial_b( (1-b) W(\tau,b)) + {\rm e}^\tau \beta ({\rm e}^\tau b) W(\tau,b) = W(\tau, b) \left[ \frac {C'(\tau)}{C(\tau)} - \mu \right].
\]

If we take into account that $\forall \tau \geq 0 \, \|W(\tau, \cdot)\|_{L^1} = 1$, by integrating the previous equation over $b \in [0,1]$, we obtain the value of $W(\tau, 0)$, hence the claimed system.
\qed
\end{proof}
The next results justify that \eqref{eq_W} is close to \eqref{eqn_w}.
\begin{lem}\label{lem_cdelta}
Under hypothesis \textbf{(H1)}, we have
$$
\lim_{\tau\rightarrow+\infty}C(\tau)\delta(\tau)=0
$$
where
\begin{equation}
C(\tau) \delta(\tau) = \frac{C'(\tau)}{C(\tau)} - \mu = \int_0^1 \left [ b {\rm e}^\tau \beta({\rm e}^\tau b) - \mu \right] W(\tau, b) {\rm d} b.
\end{equation}
\end{lem}
\begin{proof}

We recall first, by definition
$$
|C(\tau)\delta(\tau)|\leq \int_0^1 |{\rm e}^\tau b\beta({\rm e}^\tau b)-\mu|W(\tau,b)db.
$$
We can split then the integral into two parts 
$$
|C(\tau)\delta(\tau)|\leq \underbrace{\int_0^{{\rm e}^{-\tau/2}} |{\rm e}^\tau b\beta({\rm e}^\tau b)-\mu|W(\tau,b)db}_{I_1(\tau)}+\underbrace{\int_{{\rm e}^{-\tau/2}}^1 |{\rm e}^\tau b\beta({\rm e}^\tau b)-\mu|W(\tau,b)db}_{I_2(\tau)}.
$$
Firstly, we have
$$
I_2(\tau)\leq \sup_{({\rm e}^{\tau/2},+\infty)}|a \beta(a)-\mu|\rightarrow 0,\quad \tau\rightarrow+\infty.
$$
To estimate $I_1(\tau)$ we notice
$$
I_1(\tau)\leq \sup_{\RR_+}|a\beta(a)-\mu|\int_0^{{\rm e}^{-\tau/2}}W\leq \sup_{\RR_+}|a\beta(a)-\mu|\left(\|W-W_\infty\|_1+\int_0^{{\rm e}^{-\tau/2}}W_\infty\right).
$$
We already know from lemma~\ref{lemW_to_DL} $\|W-W_\infty\|_1\rightarrow 0$. Furthermore, for large $\tau$ we have 
$$
0\leq \int_0^{{\rm e}^{-\tau/2}}W_\infty\leq \frac{c_\infty}{(1-{\rm e}^{-\tau/2})^{1-\mu}}\frac{{\rm e}^{-\tau(1-\mu)/2}}{1-\mu}\rightarrow 0.
$$
Therefore, we have $\lim_{+\infty}I_1=0$, which concludes the proof of the lemma.
\qed
\end{proof}

\subsection{Dissipation of entropy with respect to $W$}
We now introduce the most important tool we will use: the relative entropy (similar to the entropy rate of a stochastic process, or the general relative entropy used in \cite{cours_entropy, Perthame}).

\begin{defn}
\label{def_entropy}
Let $w$ be a solution of the equation \eqref{eqn_w} with support included in $[0,1)$.
Let $H$ be a convex, continuous function, $C^1$ by parts, which reaches its minimum, 0, at 1. We define the generalised relative entropy as: 
\begin{equation}
\mathcal H (\tau) = \int_0^1 H \left( \frac{w(\tau,b)}{W(\tau,b)} \right) W(\tau,b) {\rm d} b .
\end{equation}
And for a non-negative measure $\nu$ on $[0,1)$ the entropy dissipation $DH(u|\nu)$ is defined by 
$$
DH(u|\nu)=\int_0^1 H(u(b))d\nu(b)-H\left(\int_0^1u(b)d\nu(b)\right).
$$
\end{defn}
Note that $DH(u|\nu)\geq 0$ if $\nu$ is a probability (by Jensen's inequality).

We are now in position to establish a first important inequality on the relative entropy
\begin{prop}
\label{prop_H'}

Under \textbf{(H1)}, the entropy $\mathcal H$ satisfies the following equality:

\begin{equation}
\label{eq_dissipation}
\mathcal H ' (\tau) = -  C(\tau) DH( u | {\rm d} \gamma_\tau) + C(\tau)\delta(\tau)\int_0^1 \left(H(u)-uH'(u)\right)W(\tau,b) db
\end{equation}
where $d\gamma_\tau(b)=\frac{{\rm e}^\tau\beta({\rm e}^\tau b)W(\tau,b)}{W(\tau,0)}$ is a non-negative measure of mass $(1+\delta(\tau))$ and $u=w/W$.
\end{prop}

\begin{proof}
The mass of $d\gamma$ is immediately derived from the equation on $W(\tau,0)$. 

We have then:
$$
\begin{cases}
		\partial_\tau w + \partial_b ( (1-b) w ) + {\rm e}^\tau \beta({\rm e}^\tau b) w = 0 \smallskip\\
		\partial_\tau W + \partial_b ( (1-b) W ) + {\rm e}^\tau \beta({\rm e}^\tau b) W = W(\tau,b) C(\tau) \delta(\tau).
	\end{cases}
$$

Denoting $u=w/W$ we arrive at 
$$\partial_\tau u + (1 - b) \partial_b u =  - C \delta u. $$
We multiply this equation by $H'(u)$ and get
$$\partial_\tau (H(u)) + (1 - b) \partial_b (H(u)) = - C \delta u H'(u).$$
$$\begin{cases}		W \partial_t (H(u)) + (1 - b) W \partial_b (H(u)) =  - W C \delta u H'(u) \medskip\\ 
		H \partial_\tau W + (1 - b) H \partial_b W + \left( {\rm e}^\tau \beta({\rm e}^\tau b) - 1 \right) H W = C \delta H W
\end{cases}
$$
$$ \partial_\tau (H(u)W) + \partial_b ((1 - b) H(u) W) + {\rm e}^\tau \beta({\rm e}^\tau b) H(u) W = \underbrace{W C \delta \left[ H(u) - u H'(u) \right]}_{\eta}.
$$
Taking the integral over $b$ of the previous expression yields:
$$ \frac{\rm d}{{\rm d} \tau} \int_0^1 H(u)W {\rm d} b - W(\tau,0) H(u(\tau,0)) + \int_0^1 {\rm e}^\tau \beta({\rm e}^\tau b)H(u)W {\rm d} b = \int_0^1 \eta {\rm d} b $$
and finally,
$$ \frac{\rm d}{{\rm d} \tau} \int_0^1 H(u)W {\rm d} b =  C(\tau) \left[ H \left( \int_0^1 \frac{{\rm e}^\tau \beta({\rm e}^\tau b) w }{W(\tau,0)} {\rm d} b \right) - \int_0^1 H(u) \frac{{\rm e}^\tau \beta({\rm e}^\tau b)W}{W(\tau,0)} {\rm d} b \right] + \int_0^1 \eta {\rm d} b
$$
which, by definition of the entropy dissipation, proves the proposition.
\qed
\end{proof}

\begin{rem}
$C(\tau) \delta(\tau)$ appears naturally as a remainder we will have to estimate in order to prove convergence-related properties for $\mathcal H$, and also for $\|W - W_\infty\|_1$.
\end{rem}

\subsection{$L^1$ convergence (without a rate) to $W$}
In this section we prove proposition~\ref{prop_H_to_0}.  We take therefore $H(x)=|x-1|$, we have then, 
$$
DH(u|d\gamma_\tau)=\int_0^1 |u-1|d\gamma_\tau-\left|\int_0^1 u d\gamma_\tau -1\right|.
$$
The core of the proof is the following
\begin{lem}\label{lem_H'_leq}
Under hypothesis \textbf{(H1)}, we have
\begin{equation}\label{eq_DHL1}
\frac{d}{d\tau}\int_0^1 |w-W|db\leq- {\rm e}^{\tau}\beta({\rm e}^\tau)\int_0^1 |w-W|db+2|C(\tau)\delta(\tau)|.
\end{equation}
\end{lem}

\begin{proof}
By definition,
$$
d\gamma_\tau =\frac{{\rm e}^\tau \beta({\rm e}^\tau b)W(\tau,b)}{W(\tau,0)}\geq \frac{{\rm e}^\tau\beta({\rm e}^\tau)}{C(\tau)}W(\tau,b)=K(\tau)W(\tau,b).
$$
Furthermore, since $\int_0^1 (u-1)W=\int_0^1 w-W=0,$ we have 
\begin{align*}
DH(u|d\gamma_\tau)&=\int_0^1 |u-1|d\gamma_\tau-\left|\int_0^1 u d\gamma_\tau -1\right|\\
&=\int_0^1 |u-1|d\gamma_\tau-\left|\int_0^1 (u-1) d\gamma_\tau -\delta\right|\\
&=\int_0^1 |u-1|d\gamma_\tau-\left|\int_0^1 (u-1) \underbrace{(d\gamma_\tau-K(\tau)W)}_{\geq 0} -\delta\right|\\
&=K(\tau)\int_0^1 |u-1|W +\int_0^1 |u-1|(d\gamma_\tau-K(\tau)W)-\left|\int_0^1 (u-1)(d\gamma_\tau-K(\tau)W)\right|  -|\delta|\\
&\geq K(\tau)\int_0^1 |u-1|W -|\delta(\tau)|.
\end{align*}
Since we also have, for $H(x)=|x-1|$,
$$
\left|H(u)-uH'(u)\right|=|-sign(u-1)|\leq 1,
$$
we obtain for this case 
$$
\left|\int_0^1 (H(u)-uH'(u)) W\right|\leq 1.
$$
And since $C(\tau)K(\tau)={\rm e}^\tau\beta({\rm e}^\tau)$, we obtain equation \eqref{eq_DHL1}.
\qed
\end{proof}

Since we already have by hypothesis ${\rm e}^\tau\beta({\rm e}^\tau)\rightarrow \mu>0$, standard ODE arguments yield
$$
\limsup_{+\infty} \int_0^1 |w-W|db\leq \frac{2\limsup_{+\infty} |C(\tau)\delta|}{\mu}.
$$
We can conclude the proof of proposition~\ref{prop_H_to_0} using lemma~\ref{lem_cdelta}. 

\begin{rem}
We let the reader check that, by defining $\underline \beta$ as in \eqref{eq_beta_bar}, we may replace the non-increasing $\beta$ hypothesis by $a \underline \beta(a) \xrightarrow[a \to \infty]{} \mu$, obtaining the following equation instead of \eqref{eq_DHL1}:
\begin{equation}\label{eq_DHL_beta_bar}
\frac{d}{d\tau}\int_0^1 |w-W|db\leq- {\rm e}^{\tau}\underline \beta({\rm e}^\tau)\int_0^1 |w-W|db+2|C(\tau)\delta(\tau)|.
\end{equation}
\end{rem}

\begin{rem}
We have now finished developing a framework which allows us to deduce the behaviour of the entropy $\mathcal H$ from suitable hypotheses made on $\beta$, and showed that, under mild conditions, the entropy tends to $0$. 
The following section will extract a convergence rate from more restrictive hypotheses.
\end{rem}

\section{Rates of convergence to the pseudo equilibrium}
\label{sec_beta}

\subsection{The key situation: $\beta (a) = \frac \mu {1+a}$}

We consider it best to start by presenting this simple case, since the following proofs contain the key innovative elements of the general case while simplifying the presentation of our results. 

Here, the pseudo equilibrium $W$ becomes, with $c(\tau) = {\rm e}^{- \mu \tau} C(\tau)$:
\[
W(\tau,b) = \frac{C(\tau)}{ (1-b)^{1-\mu} (1 + {\rm e}^{\tau} b)^\mu } = \frac{c(\tau)}{ (1-b)^{1-\mu} ({\rm e}^{-\tau} + b)^\mu }.
\]

We can now compute $\delta$ as follows.

\begin{lem}\label{lem_delta_1}
\begin{equation}
\delta(\tau) = - \frac{{\rm e}^{-\tau}}{1+{\rm e}^{-\tau}} .
\end{equation}

\begin{proof}

\[
\begin{array}{lll}
C(\tau) \delta(\tau) & = & \frac{C'(\tau)}{C(\tau)} - \mu = \frac {c'(\tau)}{c(\tau)} \smallskip\\
& = & \int_0^1 \left[ b {\rm e}^\tau \beta({\rm e}^\tau b) - \mu \right] W(\tau, b) {\rm d} b = \int_0^1 \frac \mu {{\rm e}^{-\tau} + b} \left[ b - ({\rm e}^{-\tau} + b) \right] W(\tau, b) {\rm d} b \smallskip\\
& = & - {\rm e}^{-\tau} \int_0^1 {\rm e}^\tau \beta({\rm e}^\tau b) W(\tau,b) {\rm d} b.
\end{array}
\]
By applying Proposition \ref{prop_W_tau}, we obtain:
\[
\begin{array}{lll}
C(\tau) \delta(\tau) & = & - {\rm e}^{-\tau} (1 + \delta(\tau)) W(\tau,0) \smallskip\\
	& = & - {\rm e}^{-\tau} (1 + \delta(\tau)) C(\tau)
\end{array}
\]
resulting in the claimed equality since $C$ doesn't vanish.
\qed
\end{proof}
\end{lem}

\begin{defn}
We call $c_\infty$ the limit at $\infty$ of $c(\tau)$ when such limit exists. Here, it is easy to see $c$ is a decreasing function and $c_\infty$ is well defined.
\end{defn}

\begin{lem}\label{lem_cdelta_1}
\begin{equation}
- 4 {\rm e}^{-(1-\mu) \tau} \leq C(\tau) \delta(\tau) \leq - \frac{c_\infty}{1+{\rm e}^{-\tau}} {\rm e}^{-(1-\mu) \tau} \leq 0.
\end{equation}
\end{lem}
\begin{proof}
We have $c(\tau) = {\rm e}^{-\mu \tau} C(\tau)$, thus $\frac{C'(\tau)}{C(\tau)} - \mu = \frac{c'(\tau)}{c(\tau)}$, and we establish two trivial bounds on $c(\tau)$. By definition:
\[
\frac 1 {c(\tau)} = \int_0^1 \frac 1 {({\rm e}^{-\tau} + b)^\mu (1-b)^{1-\mu}} {\rm d} b
\]
hence, since $\tau \geq 0$ and $0 \leq \mu \leq 1$, we have:
\[
\frac 1 {c(\tau)} \geq \int_0^1 (1+b)^{-\mu} (1-b)^{\mu-1} {\rm d} b 
\geq \int_0^1 \left( \frac{1-b}{1+b} \right)^\mu {\rm d} b 
\geq \int_0^1 \frac 1 2 (1-b)^\mu {\rm d} b
= \frac 1 {2 (1+\mu)} \geq \frac 1 4 \,;
\]
and likewise:
\[
\frac 1 {c(\tau)} \leq \int_0^1 b^{-\mu} (1-b)^{\mu -1} {\rm d} b = \frac 1 {c_\infty} \, .
\]
It follows that:
\begin{equation}
\label{c}
0 \leq c_\infty \leq c(\tau) \leq 2(1+\mu) \leq 4.
\end{equation}

The result of Lemma \ref{lem_delta_1} allows us to conclude.
\qed
\end{proof}

Putting together this lemma and equation \eqref{eq_DHL1} gives us

\begin{cor}\label{cor_H'_beta_ref}
The following inequality holds
\begin{equation}
\label{eqn_H'_beta_ref}
\mathcal H'(\tau) \leq - \left( \mu - \frac \mu {1+{\rm e}^\tau} \right) \mathcal H (\tau) + 8 {\rm e}^{-(1-\mu)\tau}.
\end{equation}
\end{cor}

All is ready to prove Theorem \ref{thm_beta_ref} by applying Gronwall's Lemma to the previous inequality.

\begin{proof}[Proof of Theorem \ref{thm_beta_ref}]
We set $f(s) = - \frac \mu {1 + e^s}$, which gives $\int_0^\tau f(s) {\rm d} s = -\mu \ln \left( \frac{1 + {\rm e}^{-\tau}} 2\right)$. Corollary \ref{cor_H'_beta_ref} implies that:
\[
\frac {\rm d} {{\rm d} \tau} \left[ \exp\left(\int_0^\tau (\mu + f(s)) {\rm d} s \right) \mathcal H(\tau) \right] \leq 8 {\rm e}^{-(1-\mu)\tau} \exp\left(\int_0^\tau (\mu + f(s)) {\rm d} s \right).
\]

By integrating over $\tau$, we obtain:
\[
\begin{array}{lll}
\mathcal H(\tau) & \leq & \mathcal H(0) {\rm e}^{-\int_0^\tau \mu + f} + 8 \int_0^\tau {\rm e}^{-(1-\mu)\tau'}{\rm e}^{\int_0^{\tau'} \mu + f} {\rm e}^{-\int_0^\tau \mu + f} {\rm d} \tau' \smallskip\\
	& \leq & {\rm e}^{-\mu \tau} \mathcal H(0) {\rm e}^{-\int_0^\tau f} + 8 {\rm e}^{-\mu \tau} \int_0^\tau {\rm e}^{(2\mu - 1)\tau'} {\rm e}^{- \int_{\tau'}^\tau f} {\rm d} \tau' \smallskip\\
	& \leq & {\rm e}^{-\mu \tau} \mathcal H(0) \left( \frac 2 {1 + {\rm e}^{-\tau}} \right)^\mu + 8 {\rm e}^{-\mu \tau} \int_0^\tau {\rm e}^{(2\mu - 1) \tau'} \left( \frac{1 + {\rm e}^{-\tau'}}{1 + {\rm e}^\tau} \right)^\mu {\rm d} \tau' \smallskip\\
	& \leq & {\rm e}^{-\mu \tau} \left( \frac 2 {1 + {\rm e}^{-\tau}} \right)^\mu \left[ \mathcal H(0)  + 8 \int_0^\tau {\rm e}^{(2\mu - 1) \tau'} {\rm d} \tau' \right].
\end{array}
\]
\qed
\end{proof}


\subsection{A larger class of $\beta$}

We now consider $\beta$ satisfying \textbf{(H2)}.

\begin{rem}
The bound on $G$ is not necessary to prove lemmas \ref{lem_WW} and \ref{lem_Cdelta_gral_beta}: $g \in L^1$ is a strong enough hypothesis. However, that precise bound in necessary for our convergence rate estimates and as such, we assume it holds throughout the section.
\end{rem}

We take the following notations:
\[
\begin{array}{lll}
G(a) & = & \int_a^\infty | g(s) | {\rm d} s \\
W_{ref}(\tau,b) & = & \frac{C_{ref}(\tau)}{(1-b)^{1-\mu} (1 + {\rm e}^\tau b)^\mu } \\
W (\tau,b) & = & \frac{C(\tau) \exp \left(- \int_0^{{\rm e}^\tau b} g(s){\rm d} s \right) }{(1-b)^{1-\mu} (1 + {\rm e}^\tau b)^\mu}  \\
\end{array}
\]
where $C_{ref}$ and $C$ ensure $\| W_{ref}(\tau) \|_1 = \| W(\tau) \| = 1$. (We use the same notation as in the proof of lemma~\ref{lemW_to_DL}).

We have:
\[
\begin{array}{lll}
C(\tau)\delta(\tau) & = & \int_0^1 \left[ {\rm e}^\tau b \beta({\rm e}^\tau b) - \mu \right] W(\tau, b) {\rm d} b \smallskip\\
		& = & \int_0^1 {\rm e}^\tau b g({\rm e}^\tau b) W {\rm d} b + \mu \int_0^1 \left[ \frac {{\rm e}^\tau b}{1+ {\rm e}^\tau b} - 1 \right] W (\tau,b) {\rm d} b.
\end{array}
\]

\begin{lem}\label{lem_WW}
Assume \textbf{(H2)} holds. Then there exists $K > 0$ such that for any $\tau \geq 0$ and $0 \leq b \leq 1$:
\begin{equation}
\frac{K^{-1}{\rm e}^{\mu \tau}}{ (1+{\rm e}^\tau b)^\mu (1-b)^{1-\mu}} \leq W(\tau,b) \leq \frac{K {\rm e}^{\mu \tau}}{ (1+{\rm e}^\tau b)^\mu (1-b)^{1-\mu}}.
\end{equation}

\end{lem}
\begin{proof}
We have:
\[
\frac{W(\tau, b)}{W_{ref}(\tau,b)} = \frac{C(\tau)}{C_{ref}(\tau)} \exp \left( -\int_0^{{\rm e}^\tau b} g(s){\rm d} s \right) \in \frac{C(\tau)}{C_{ref}(\tau)} \left[ {\rm e}^{-\|g\|_1} , {\rm e}^{\|g\|_1} \right].
\]
And since $\int_0^1 W(\tau,b) {\rm d} b = \int_0^1 W_{ref} (\tau,b) {\rm d} b = 1$, it follows that:
\[
{\rm e}^{-\|g\|_1} \leq \frac{C(\tau)}{C_{ref}(\tau)} \leq {\rm e}^{\|g\|_1}
\]
which gives us $\frac{W}{W_{ref}} \in L^\infty$ in the sense given above.
\qed
\end{proof}

This result leads, through a proof analogous to that of lemma~\ref{lem_cdelta_1}, to the following
\begin{lem}\label{lem_Cdelta_gral_beta}
Under hypothesis \textbf{(H2)}, there exists a positive $M$ such that: 
\begin{equation}
\left| \mu \int_0^1 \left[ \frac{{\rm e}^\tau b}{1 + {\rm e}^\tau b} - 1 \right] W(\tau, b) {\rm d} b \right| \leq M {\rm e}^{-(1-\mu)\tau}.
\end{equation}
\end{lem}

We now give the strategy for estimating the rate of convergence. It is based on the same procedure as before. Consider a non-increasing $\beta$ and equation \eqref{eq_DHL1}, which measures the dissipation of entropy (or equation \eqref{eq_DHL_beta_bar} under the corresponding hypothesis).
$$
\mathcal{H}'(\tau)\leq -{\rm e}^\tau\beta({\rm e}^\tau)\mathcal{H}(\tau)+2|C(\tau)\delta(\tau)|.
$$
We have then
\begin{equation}
\label{eq_HeB}
\left(\mathcal{H}{\rm e}^{B({\rm e}^\tau)}\right)'(\tau)\leq 2|C(\tau)\delta(\tau)|{\rm e}^{B({\rm e}^\tau)}.
\end{equation}
We  recall
$$
C(\tau)\delta(\tau)=\int_0^1 ({\rm e}^\tau b\beta({\rm e}^\tau b)-\mu)W(\tau,b)db.
$$ 
And also by the definition of $W$ and lemma~\ref{lem_WW}
$$
W(\tau,b)\leq \frac{K}{({\rm e}^{-\tau}+b)^\mu(1-b)^{1-\mu}}.
$$
Therefore, we easily obtain
$$
|C(\tau)\delta(\tau)|\leq K\left(\int_0^1 \frac{{\rm e}^\tau b|g|({\rm e}^\tau b)}{({\rm e}^{-\tau}+b)^\mu(1-b)^{1-\mu}}db+ \left| \mu \int_0^1  \left(\frac{{\rm e}^\tau b}{1 + {\rm e}^\tau b} - 1 \right) W(\tau, b) {\rm d} b \right|\right).
$$
Lemma~\ref{lem_Cdelta_gral_beta} then gives us
$$
|C(\tau)\delta(\tau)| \leq K\left(\int_0^1 \frac{{\rm e}^\tau b|g|({\rm e}^\tau b)}{({\rm e}^{-\tau}+b)^\mu(1-b)^{1-\mu}}db+ Me^{(\mu-1)\tau}\right).
$$
The constant $K$ may change value from line to line. 

We integrate equation \eqref{eq_HeB} and get 
$$
\mathcal{H}(\tau){\rm e}^{B({\rm e}^\tau)}\leq \mathcal{H}(0)+K \int_0^\tau {\rm e}^{B({\rm e}^{\tau'})}\int_0^1 \frac{{\rm e}^{\tau'} b|g|({\rm e}^{\tau'} b)}{({\rm e}^{-\tau'}+b)^\mu(1-b)^{1-\mu}}dbd\tau' +K \int_0^\tau {\rm e}^{(\mu-1)\tau'+B({\rm e}^{\tau'})}d\tau'.
$$
Using the fact the $B({\rm e}^\tau)-\mu \tau$ is bounded from above and below, we can replace $B({\rm e}^\tau)$ by $\mu \tau$ with just a change of constants.
\begin{equation}\label{eq_g_general}
\mathcal{H}(\tau){\rm e}^{\mu\tau}\leq K\left(1+ \int_0^\tau \int_0^1 \frac{{\rm e}^{\mu\tau'}{\rm e}^{\tau'} b|g|({\rm e}^{\tau'} b)}{({\rm e}^{-\tau'}+b)^\mu(1-b)^{1-\mu}}dbd\tau' +\int_0^\tau {\rm e}^{(2 \mu-1)\tau'}d\tau'\right).
\end{equation}
\begin{rem}
We let the reader check that the non-increasing $\beta$ hypothesis may be replaced by the following condition on the function defined in \eqref{eq_beta_bar}: $a \underline \beta(a) \xrightarrow[a \to \infty]{} \mu$. Replacing the use of $B$ by that of $\underline B (a) = \int_0^a \underline \beta$, we still obtain equation \eqref{eq_g_general} up to multiplication by a constant, since $\underline B({\rm e}^\tau) - \mu \tau$ is also bounded from above and below.
\end{rem}
Now the work is focused on the estimate of the middle quantity 
$$
I(\tau)=\int_0^\tau \int_0^1\frac{{\rm e}^{\mu\tau'}{\rm e}^{\tau'} b|g|({\rm e}^{\tau'} b)}{({\rm e}^{-\tau'}+b)^\mu(1-b)^{1-\mu}}dbd\tau' 
$$
We integrate by parts with respect to $\tau'$ and recall that $\frac{d}{d\tau}G({\rm e}^\tau b)=-{\rm e}^\tau b |g({\rm e}^\tau b)|$. We have
\begin{eqnarray*}
I(\tau)&=&\int_0^1 \left[-G({\rm e}^{\tau'} b)\frac{{\rm e}^{\mu\tau'}}{({\rm e}^{-\tau'}+b)^\mu(1-b)^{1-\mu}}\right]^\tau_0 db\\ &+&\int_0^1\int_0^\tau G({\rm e}^{\tau' }b)\left(\mu\frac{{\rm e}^{\mu\tau'}}{({\rm e}^{-\tau'}+b)^\mu(1-b)^{1-\mu}}+\frac{\mu {\rm e}^{-\tau}{\rm e}^{\mu\tau'}}{({\rm e}^{-\tau'}+b)^{\mu+1}(1-b)^{1-\mu}}\right)d\tau'db.
\end{eqnarray*}
The first term is bounded from above (since $G\geq 0$) by 
$$
\int_0^1 \frac{\|g\|_1}{(1+b)^\mu (1-b)^{1-\mu}}.
$$
Finally, since $\frac{{\rm e}^{-\tau}}{{\rm e}^{-\tau}+b}\leq 1$,
$$
I(\tau)\leq K + 2\mu\int_0^1\int_0^\tau G({\rm e}^{\tau' }b)\frac{{\rm e}^{\mu\tau'}}{({\rm e}^{-\tau'}+b)^\mu(1-b)^{1-\mu}}db
$$
We need a sharp estimate on the second term.
We focus our efforts on the case 
$$
G(a)\leq \frac{K}{(1+a)^\alpha}\quad \text{for some }\alpha >0.
$$
\begin{prop}\label{prop_conv_g}
Assume hypothesis \textbf{(H2)}.\\
Then, if $\mu\not=1/2$
$$
\mathcal{H}(\tau)\leq 
\begin{cases}
K((1+\tau){\rm e}^{-\mu\tau}), \quad & \text{ if } \mu=\alpha <1-\mu\,  \\[0.3cm]
K({\rm e}^{-\alpha \tau}+{\rm e}^{-\mu\tau}), \quad & \text{ if } \mu\not=\alpha <1-\mu\,  \\[0.3cm]
K({\rm e}^{-(1-\mu) \tau}+{\rm e}^{-\mu\tau}),\quad & \text{ if } \alpha >1-\mu \\[0.3cm]
K(\tau {\rm e}^{-(1-\mu) \tau}+{\rm e}^{-\mu\tau}),\quad & \text{ if } \alpha=1-\mu.
 \end{cases}
$$
If $\mu=1/2$, then
$$
\mathcal{H}(\tau)\leq 
\begin{cases}
K({\rm e}^{-\alpha \tau}+{\rm e}^{-\mu\tau}), \quad & \text{ if } \alpha <1/2\,  \\[0.3cm]
K({\rm e}^{-\tau/2}), \quad & \text{ if } \alpha >1/2 \\[0.3cm]
K((1+\tau^2) {\rm e}^{-\tau/2}), \quad & \text{ if } \alpha=1/2.
 \end{cases}
$$
\end{prop}
Note that for $\alpha >(1-\mu)$ the rate is the same than the one for $g=0$.
\begin{proof}

We need an estimate of 
$$
\int_0^1\int_0^\tau \frac{1}{(1+{\rm e}^{\tau'} b)^\alpha}\frac{{\rm e}^{\mu\tau'}}{({\rm e}^{-\tau'}+b)^\mu(1-b)^{1-\mu}}d\tau'db=\int_0^1\int_0^\tau \frac{{\rm e}^{(\mu-\alpha)\tau'}}{({\rm e}^{-\tau'}+b)^{\mu+\alpha}(1-b)^{1-\mu}}d\tau'db
$$
Putting all together with \eqref{eq_g_general}, we have 
$$
H(\tau)\leq K\left({\rm e}^{-\mu\tau}+{\rm e}^{-\mu\tau}\int_0^1\int_0^\tau \frac{{\rm e}^{(\mu-\alpha)\tau'}}{({\rm e}^{-\tau'}+b)^{\mu+\alpha}(1-b)^{1-\mu}}d\tau'db+{\rm e}^{-\mu\tau}\int_0^\tau {\rm e}^{(2\mu-1)\tau'}d\tau'\right).
$$
We essentially need to estimate the middle term.
\begin{lem}\label{lem_I}
Under \textbf{(H2)}, the following holds true:
$$
\int_0^1\int_0^\tau \mu\frac{{\rm e}^{(\mu-\alpha)\tau'}}{({\rm e}^{-\tau'}+b)^{\mu+\alpha}(1-b)^{1-\mu}}d\tau'db \leq 
\begin{cases} 
K\int_0^\tau {\rm e}^{(\mu-\alpha)\tau'}d\tau' \;, & \text{ if } \alpha <1-\mu,\\[0.3cm]
K\int_0^\tau {\rm e}^{(2\mu-1)\tau'}d\tau' \;, & \text{ if }\alpha>1-\mu,\\[0.3cm]
K\int_0^\tau (1+\tau'){\rm e}^{(2\mu-1)\tau'}d\tau' \;, & \text{ if }\alpha=1-\mu.\\[0.3cm]
\end{cases}
$$
\end{lem}

\begin{proof}[Proof of the lemma]
\textbf{Case 1: $\alpha<1-\mu$}.\\
For $\alpha <1-\mu$, we simply use the fact that 
$$
\int_0^1 \frac{1}{({\rm e}^{-\tau'}+b)^{\mu+\alpha}(1-b)^{1-\mu}}db\leq \int_0^1 \frac{1}{b^{\mu+\alpha}(1-b)^{1-\mu}}db<+\infty
$$

\textbf{Case 2: $\alpha >1-\mu$.}\\
It is obvious that we can restrict to the case $\alpha<1$. 
We first need a few intermediate computations. Firstly, for $\gamma>0$, we have
$$
\int_0^1 \frac{\gamma}{({\rm e}^{-\tau}+b)^{\gamma+1}(1-b)^{1-\gamma}}=\frac{1}{1+{\rm e}^{-\tau}}\int_0^1 -\frac{d}{db} f(b) {\rm e}^{f(b)}db.
$$
Where 
$$
f(b)=-\gamma \ln({\rm e}^{-\tau}+b)+\gamma\ln(1-b).
$$
Therefore, we have
\begin{equation}\label{eq_muplusun}
\int_0^1 \frac{\gamma}{({\rm e}^\tau+b)^{\gamma+1}(1-b)^{1-\gamma}}=\frac{1}{1+{\rm e}^{-\tau}}\left[-\left(\frac{(1-b)}{{\rm e}^{-\tau}+b}\right)^\gamma\right]^1_0=\frac{{\rm e}^{\gamma\tau}}{1+{\rm e}^{-\tau}}.
\end{equation}
Using this computation and noticing $\gamma =\mu+\alpha-1$, we can easily establish that, for $\alpha <1$, 
$$\int_0^1 \frac{1}{({\rm e}^{-\tau}+b)^{\mu+\alpha}(1-b)^{1-\mu}}db=\int_0^1 \frac{(1-b)^{1-\alpha}}{({\rm e}^{-\tau}+b)^{\gamma +1}(1-b)^{1-\gamma}}db. $$
Applying then \eqref{eq_muplusun} and using $(1-b)^{1-\alpha}\leq 1$, we arrive at 
$$
\int_0^1 \frac{1}{({\rm e}^{-\tau}+b)^{\mu+\alpha}(1-b)^{1-\mu}}db\leq Ke^{\gamma \tau}.
$$
Injecting, we obtain 
$$
I_\alpha(\tau)\leq K\int_0^\tau {\rm e}^{(\gamma+\mu-\alpha)\tau'}d\tau'=K\int_0^\tau {\rm e}^{(2\mu-1)\tau'}d\tau'.
$$

\textbf{Case 3: $\alpha=1-\mu$.}\\
In this case 
$$
I_\alpha(\tau)=\int_0^\tau\int_0^1\frac{{\rm e}^{(2\mu-1)\tau'}}{({\rm e}^{-\tau'}+b)(1-b)^{1-\mu}}db d\tau',
$$
Cutting the integral on $b$ at $1/2$ for instance, it is easy to establish 
$$
I_\alpha(\tau)\leq K\left(\int_0^\tau\int_0^1 {\rm e}^{(2\mu-1)\tau'}\left(1+\frac{1}{(e{-\tau'}+b)}\right)db d\tau'\right)\leq K\left(\int_0^\tau {\rm e}^{(2\mu-1)\tau'}\left(1+\log(1+{\rm e}^{\tau'})\right) d\tau'\right),
$$
thereby, we have
$$
I_\alpha(\tau)\leq K\int_0^\tau (2+\tau'){\rm e}^{(2\mu-1)\tau'}d\tau'.
$$
This ends the proof of lemma~\ref{lem_I}. 
\qed
\end{proof}

To end the proof of proposition \ref{prop_conv_g}, we essentially just need to discuss whether the integrals of type $\int_0^\tau {\rm e}^{\lambda \tau'} d\tau'$ take value $\frac{{\rm e}^{\lambda\tau}-1}{\lambda}$ or $\tau$, and similarly for integrals of type  $\int_0^\tau \tau'{\rm e}^{\lambda \tau'} d\tau'.$
\qed 
\end{proof}
\section{Rates of convergence towards the equilibrium $W_\infty$. }
\label{sec_other_CR}

\subsection{Quantitative estimate of $\|W-W_\infty\|_1$}

In what follows, we justify how the rate of convergence of $w$ to $W$ can be extended to quantify (up to a multiplicative constant) the rate of convergence towards $W_\infty$. The main remark is the following. 
\begin{lem}\label{lem_W_W_infty}
Under hypothesis \textbf{(H2)}, we have
$$
\|W-W_\infty\|_1\leq 2\int_0^1 |{\rm e}^\tau b\beta({\rm e}^\tau b)-\mu|W(\tau, b) {\rm d} b.
$$
\end{lem}
\begin{proof}
We already know from lemma~\ref{lemW_to_DL} $\lim_{+\infty}\|W(\tau,\cdot)-W_\infty\|_1=0$. Therefore
$$
\|W(\tau,\cdot)-W_\infty\|_1\leq \int_\tau^\infty \left|\frac{d}{d\tau}\int_0^1 |W(\tau',b)-W_\infty(b)|db \right|d\tau'\leq \int_\tau^\infty\int_0^1 |\partial_\tau W(\tau',b)|dbd\tau'.
$$
Since we have
$$
\partial_\tau W=\frac{C'(\tau)}{C(\tau)}W-{\rm e}^\tau b\beta({\rm e}^\tau b) W=\left(\frac{C'(\tau)}{C(\tau)}-\mu\right) W+ (\mu-{\rm e}^\tau b\beta({\rm e}^\tau b)) W,
$$
it follows that
$$
\int_0^1 |\partial_\tau W|\leq \left|\frac{C'(\tau)}{C(\tau)}-\mu\right|+\int_0^1|{\rm e}^\tau b\beta({\rm e}^\tau b)-\mu| W.
$$
And since
$$
\left|\frac{C'(\tau)}{C(\tau)}-\mu\right|\leq \int_0^1|{\rm e}^\tau b\beta({\rm e}^\tau b)-\mu| W,
$$
we can conclude.
\qed
\end{proof}
\begin{rem}
The bound on $G$ has not been used in the proof of the previous lemma, for which $g \in L^1$ is a strong enough hypothesis.
\end{rem}
We encounter yet again the quantity $\int_0^1|{\rm e}^\tau b\beta({\rm e}^\tau b)-\mu| W$, for which we have already given a time-weighted average estimate in the reasoning following equation \eqref{eq_g_general}. Let us now provide a pointwise estimate. In the reference case, we have
$$
\|W(\tau,\cdot)-W_\infty\|_1\leq K\int_\tau^\infty \frac{{\rm e}^{(\mu-1)\tau'}}{1+{\rm e}^{-\tau'}}d\tau' \leq K {\rm e}^{(\mu-1)\tau}
$$
In the situation described by proposition~\ref{prop_conv_g}, with $\beta=\frac{\mu}{1+a}+g(a)$ and for some $\alpha >0$, we have,
$$\int_a^\infty |g|\leq \frac{K}{(1+a)^\alpha}.$$
In this case, we can split $\int_0^1|{\rm e}^\tau b\beta({\rm e}^\tau b)-\mu| W$ into two parts and use the previous arguments of the proof of lemma~\ref{lem_Cdelta_gral_beta} to claim 
\begin{align*}
\int_0^1 \left|{\rm e}^\tau b\beta({\rm e}^\tau b)-\mu \right| W\leq \int_\tau^\infty\int_0^1 \left|\frac{\mu b}{{\rm e}^{-\tau'}+b}-\mu \right|\frac{1}{({\rm e}^{-\tau'}+b)^\mu(1-b)^{1-\mu}}dbd\tau' \\+K\int_\tau^\infty\int_0^1|{\rm e}^{\tau'} b g({\rm e}^{\tau'} b)|\frac{1}{({\rm e}^{-\tau'}+b)^\mu(1-b)^{1-\mu}}dbd\tau' 
\end{align*}
The first term is already know to be bounded by $Ke^{(\mu-1)\tau}$ by lemma~\ref{lem_cdelta_1} . The second term satisfies

\[
\begin{array}{lll}
\displaystyle{ \int_\tau^\infty\int_0^1|{\rm e}^{\tau'} b g({\rm e}^{\tau'} b)|\frac{1}{({\rm e}^{-\tau'}+b)^\mu(1-b)^{1-\mu}} {\rm d} b {\rm d} \tau' }& \displaystyle{\leq} & \displaystyle{\int_\tau^\infty\int_0^1|{\rm e}^{\tau'} b g({\rm e}^{\tau'} b)|\frac{1}{b^\mu(1-b)^{1-\mu}}{\rm d} b {\rm d}\tau'} \medskip\\
    &\displaystyle{ \leq }&\displaystyle{ \int_0^1\frac{1}{b^\mu(1-b)^{1-\mu}}\left(\int_{{\rm e}^\tau b}^\infty |g| \right) {\rm d} b } \medskip\\
&\displaystyle{ \leq }&\displaystyle{ \int_0^1 \frac{1}{b^\mu(1-b)^{1-\mu}} \frac{1}{(1+{\rm e}^\tau b)^\alpha} {\rm d} b } \medskip\\
&\displaystyle{ \leq }&\displaystyle{ Ke^{-\alpha\tau} }
\end{array}
\]

Hence the rate of convergence $\|W(\tau, \cdot) - W_\infty\|_1 \leq K \left( {\rm e}^{(\mu-1) \tau} + {\rm e}^{-\alpha \tau} \right)$.

\subsection{Possible influence of the initial condition}

Let us prove a lower bound on the convergence rate of $\ln \left( \| w(\tau, \cdot) - W(\tau, \cdot) \|_{L^1(0,1)} \right)$ for an initial age distribution $w(0, b) = \delta_0 (b)$.
\begin{prop}
\label{prop_lower_bound}
Consider the reference case
\[
\beta(a) = \frac \mu {1 + a}.
\]
Suppose the initial age distribution satisfies:
\[
w^0(b) = \delta_0(b).
\]
We can bound below the total variation:
\begin{equation}
\|W(\tau, b) - w(\tau,b)\|_{TV} \geq {\rm e}^{-\mu \tau}.
\end{equation}
\end{prop}
\begin{proof}
For
\[
\left\{
\begin{array}{lll}
\phi (\tau) & = & 1 - {\rm e}^{-\tau} \\
\xi (\tau) & = & \exp \left( \int_0^\tau e^s \beta(e^s \phi(s)) {\rm d} s \right)
\end{array}
\right.
\]
we have:
\[
\begin{array}{lll}
\frac{\rm d}{{\rm d} \tau} \left[ \xi(\tau) w(\tau, \phi(\tau) \right] & = & \xi \left[ \frac{\xi '}{\xi} w + \partial_\tau w + \phi' \partial_b w \right]  \smallskip\\
	& = & \xi \left[ \partial_\tau w + \partial_b \left( (1 - \phi) w \right) + {\rm e}^\tau \beta({\rm e}^\tau \phi) w \right] = 0.
\end{array}
\]

It follows that 
\[
w(\tau, \phi(\tau)) = \exp \left( - \int_0^\tau e^s \beta(e^s \phi(s)) {\rm d} s \right) w(0, \phi(0))
\]
which, after injecting the corresponding values, yields
\[
w(\tau, 1 - {\rm e}^{-\tau}) = {\rm e}^{-\mu \tau} \delta
\]
with $\delta$ a Dirac mass. Therefore :
\[
w(\tau, \cdot) = {\rm e}^{-\mu \tau} \delta_{ 1 - {\rm e}^{-\tau} } + w_j
\]
where $w_j \geq 0$ is the distribution of particles that have jumped at least once over $(0, \tau]$.

Since $W(\tau,b) {\rm d} b$ is an atomless measure, any Dirac mass and $W {\rm d} b$ are stranger measures, hence:
\[
\|W(\tau, b) - w(\tau,b)\|_{TV} \geq {\rm e}^{-\mu \tau}.
\]
\qed
\end{proof}

For $\mu < \frac 1 2$, this lower bound agrees up to multiplication by a constant with the upper bound given in theorem \ref{thm_beta_ref}: our convergence exponent is optimal for $\mu < \frac 1 2$.

\begin{rem}
It is worth noting that we have the trivial bound:
\begin{equation}
\mathcal H(0) \leq 2.
\end{equation}
($\mathcal H(0)$ can be greater than $1$ if $w^0(b){\rm d} b$ has atoms.)
\end{rem}

\begin{rem}
Our results are proved for compactly-supported initial age distributions, and they will most likely hold for initial age distributions that decrease fast enough. However, if this is not the case, the convergence rates might be affected in a way left for future investigation.
\end{rem}

\subsection{Convergence rates for natural variables}

We recall:
\[
n(t,a) = {\rm e}^{-\tau} w(\tau, b)
\]
where
\[
\left\{
\begin{array}{lll}
\tau & = & \ln ( 1 + t )\\
b & = & \frac a {1+t}.
\end{array}
\right.
\]

\begin{defn}
\label{def_N}
We set:
\begin{equation}
N(t,a) = {\rm e}^{-\tau} W(\tau,b) =  \frac {c(\ln(1+t))}{{(1+t)^2} (1+a)^\mu (1+t-a)^{1-\mu}}
\end{equation}
\end{defn}

which leads to the following Proposition.

\begin{prop}
If
\[
\| w(\tau, \cdot) - W(\tau, \cdot) \|_{L^1([0,1])} \leq K_1 {\rm e}^{-\mu \tau} + K_2 {\rm e}^{-(1-\mu)\tau}
\]
then:
\begin{equation}
\| n(t, \cdot) - N(t, \cdot) \|_{L^1(\RR_+)} \leq \frac {K_1}{(1+t)^\mu} + \frac{K_2}{(1+t)^{1-\mu}}.
\end{equation}
\end{prop}

\begin{proof}
The ${\rm e}^{-\tau}$ appearing as the Jacobian of the change of integration variables is compensated by the ${\rm e}^{-\tau}$ in the definition of $w$ and we get the claimed result.
\qed
\end{proof}

Therefore, in the reference case $\beta (a) = \frac \mu {1+a}$, the distribution of walkers that have age $a$ at time $t$ converges to $N(t,a)$ algebraically fast, with a rate that is essentially given by $t^{-\mathrm{min}\{\mu,1-\mu\}}$: we recover corollary \ref{cor_N}.

\section{Monte Carlo simulations}
\label{sec_simulations}

In order to illustrate the evolution of the age distribution of the system and check the accuracy of the convergence rates to self-similar equilibrium, we have carried out Monte-Carlo simulations for our reference case $\beta (a) = \frac \mu {1+a}$.

In these simulations, we describe explicitly each individual walker $i$ by associating it with an age $a_i$ and a first jumping time $\tau_i$. The initial age of each walker is chosen according to some initial distribution, for instance uniform distribution in $[0,1]$ or a Dirac distribution at age $a=0$. The first jumping time of each random walker is sampled from the distribution $\phi(\tau)=\mu/(1+\tau)^{1+\mu}$, that corresponds to our reference jump rate $\beta(a)=\mu/(1+a)$. The simulation then iterates the following steps: (\textit{i}) find $k$, the walker with the earlier jump time: $k=\underset{i}{\arg\min} \tau_i$, then  (\textit{ii}) make it jump, i.e. reset its age $a_k=0$ and finally, (\textit{iii}) pick its next jump time $\tau_k$ according to $\phi(\tau)$.
During the simulation, we store the distance between the dynamic equilibrium $W$ at that time and the observed distribution of rescaled ages $b_i=a_i/(1+t)$ of all the walkers $i$ in the simulation: $\| w(\tau, \cdot) - W(\tau, \cdot) \|_{L^1([0,1])}$. We also compute at each time step the $L^1$ norm of the difference $ w(\tau, \cdot) - W_\infty$. Unless stated otherwise we use 20,000 random walkers in each simulation.

First, we note that in all cases, the simulated $L^1$ distance between $w$ and the pseudo-equilibrium $W$ is indeed bounded above by the expression given in theorem \ref{thm_beta_ref} (except at very high $\tau$, when our bound becomes lower than the numerical error of the simulation). The example given in figure \ref{fig_upper_bound} corresponds to $\mu = 0,4$ and $\mu = 0,8$, and an initial age distribution Dirac at 0 for the red dots, and uniform on $[0,1)$ for the blue dots, the black curve representing the upper bound proved in theorem \ref{thm_beta_ref} taken for $\mathcal H(0) =2$, which is an upper bound for $\mathcal H(0)$. As we see, the multiplicative constant we lose (the overestimation of $K$ in theorem \ref{thm_beta_ref} corresponding to the losses throughout the inequalities used to prove our bound) is not too high.

\begin{figure}[H]
\begin{minipage}{\textwidth}
\centering{\includegraphics[width=\linewidth]{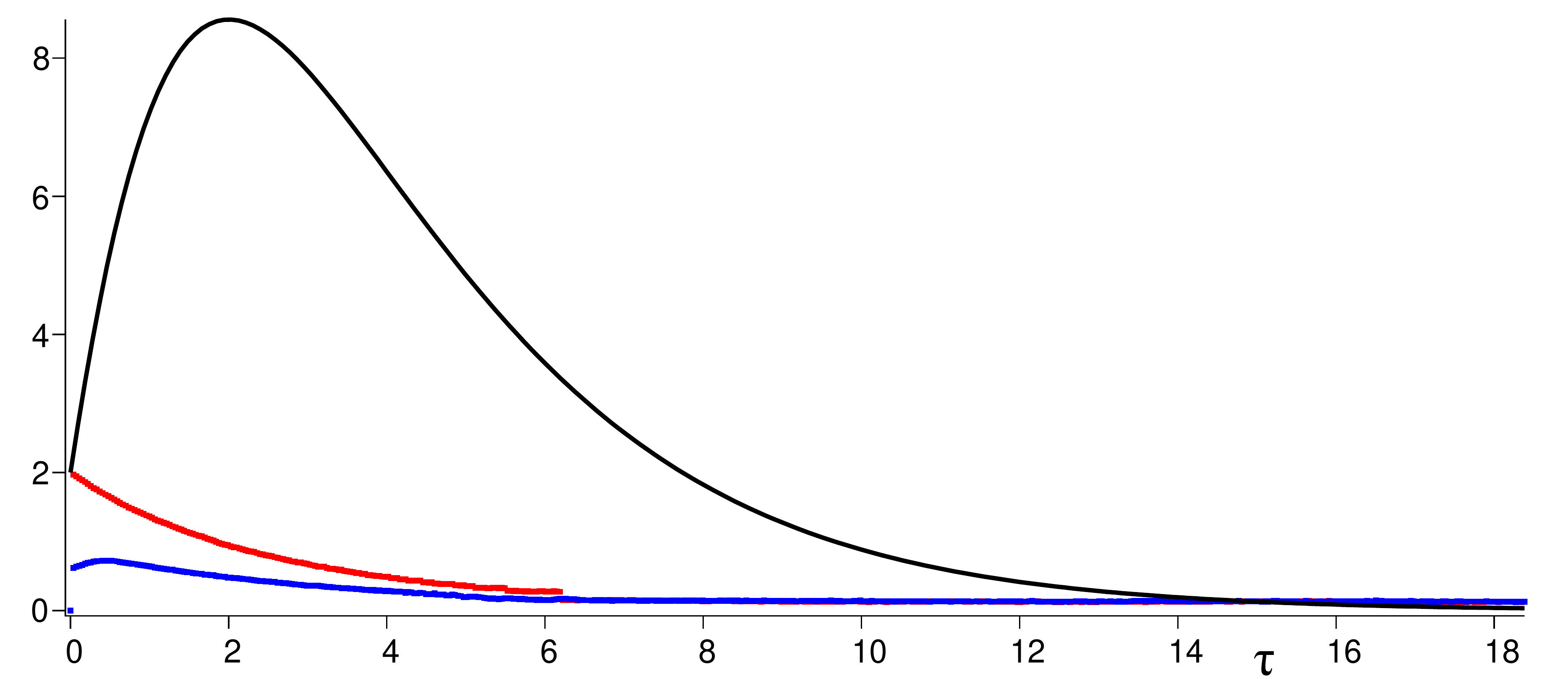}} 
$\mu = 0,4$
\end{minipage}
\begin{minipage}{\textwidth}
\centering{\includegraphics[width=\linewidth]{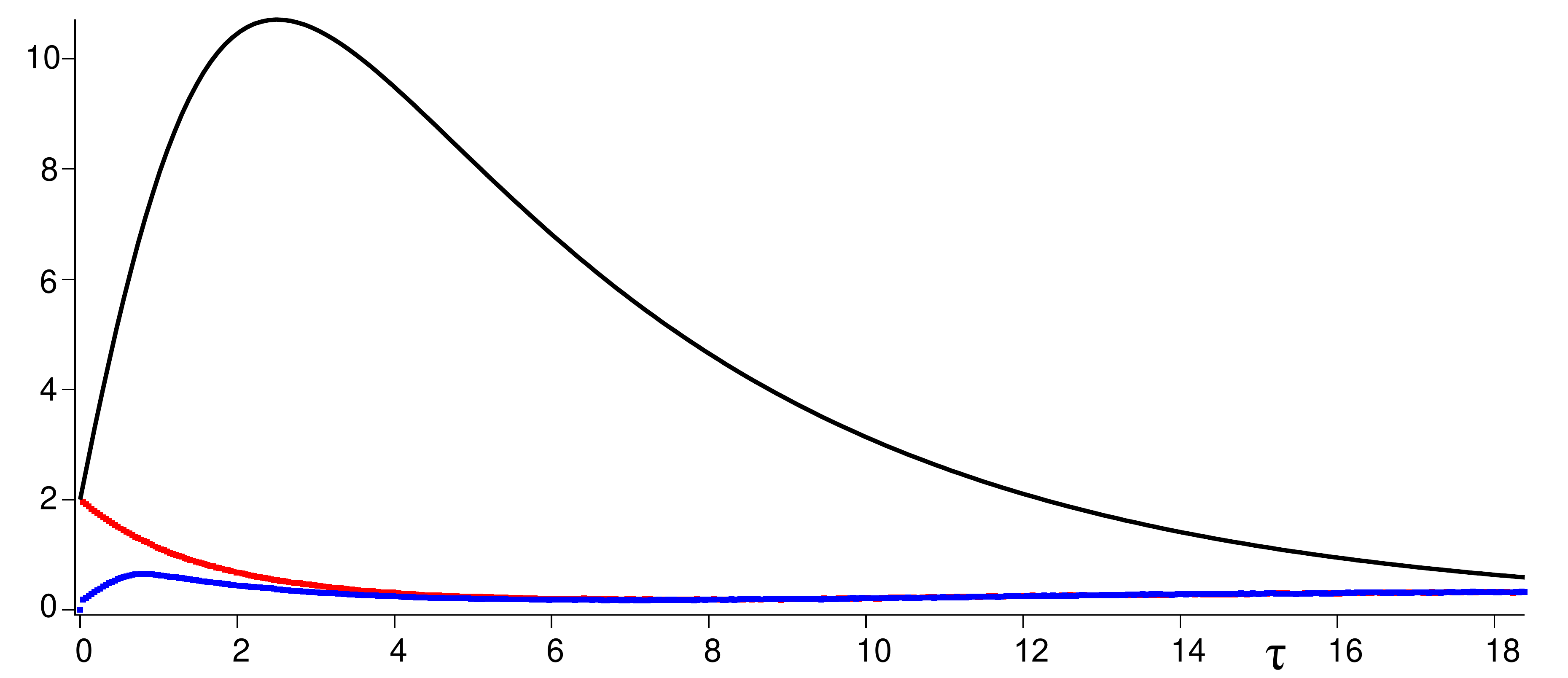}} 
$\mu = 0,8$
\end{minipage}
\caption{$\| w(\tau, \cdot) - W(\tau, \cdot) \|_1$ lies under the theoretical bound (black curve) 
for an initial age distribution $\delta_0$ (red dots) and $\mathcal U (0,1)$ (blue dots).}
\label{fig_upper_bound}
\end{figure}

In order to illustrate graphically the behaviour of the solution to our equations and its convergence towards the pseudo-equilibrium, Figure \ref{fig_n_w} displays, for $\mu = 0,6$ and an initial age distribution $n^0 = w^0 = \delta_0$, the time evolution of the simulation results expressed either in the original variables $n(t, \cdot)$ (histograms), $N(t, \cdot)$ (full line) on the left-hand side column or in the rescaled variables $w(\tau, \cdot)$ (histograms), $W(\tau, \cdot)$ (full line) on the right-hand side column. Moreover, the rescaled variables panels also show as grey dotted lines the equilibrium $W_\infty$, to which $W$ converges.

\begin{figure}[H]
\centering{\includegraphics[width=\linewidth]{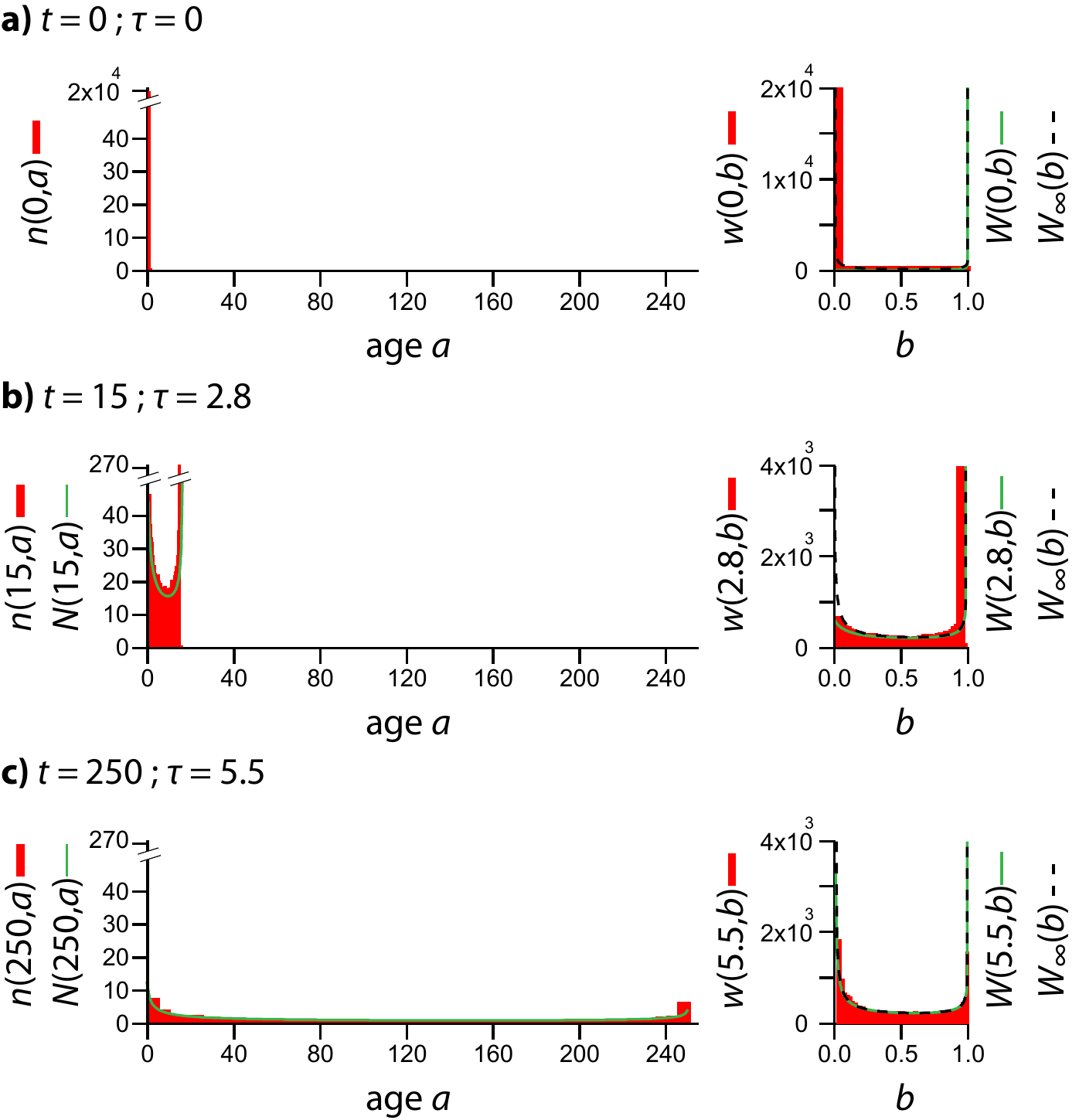}}
\caption{Evolution of $n$, $N$, $w$, $W$ and $W_\infty$ along time, for $\mu = 0,6$ and an initial age distribution $n^0 = w^0 = \delta_0$.}
\label{fig_n_w}
\end{figure}

From visual inspection of the is figure, it is clear that $n(t, a)$ largely flattens as $t \to \infty$ (note the difference in the y-axis scale between the panels). The figure depicts a pointwise convergence of the simulated $w$ to the pseudo-equilibrium $W$ which in turn converges pointwise to $W_\infty$. Moreover, it illustrates how rescaling allows a better description of the self-similar behaviour, which is difficult to grasp in natural variables since $n$ converges pointwise to $0$. The next sections quantify the simulated convergence rates.

\subsection{Exponential fit of $\| w(\tau, \cdot) - W(\tau, \cdot) \|_1$}

To quantify the convergence rates in the simulations, we fit the distance $\| w(\tau, \cdot) - W(\tau, \cdot) \|$ by the following function:

\begin{equation}
\label{eqn_f}
f(\lambda, A,B) = A {\rm e}^{-\lambda \tau} + B {\rm e}^{-(1-\lambda)\tau} + C.
\end{equation}

\begin{rem}[Heuristic estimate of the error term]
$C$ is a simulation error, that we evaluate to $C \approx 0,1$. This is consistent both with empirical evidence and with a simple heuristic overevaluation of $C$ as $ \sqrt{\# {\text bins} / \# {\text particles}}$, which is roughly $0,16$.
\end{rem}

\begin{rem}
According to the above analysis one expects $\lambda=\mu$. $A$ and $B$ are multiplicative parameters: we expect $A$ around $2$ and $|B|$ close to $0$, since $\mathcal H (0) = 2$ and our upper boundary is of the form $\mathcal H(\tau) \leq \left[ \mathcal H(0) - (\leq 0) \right] {\rm e}^{-\mu\tau} + k {\rm e}^{-(1-\mu)\tau}$, with $k$ small.
\end{rem}

\begin{rem}
Another possible explanation of the predominance of ${\rm e}^{-\mu \tau}$ over ${\rm e}^{-(1-\mu) \tau}$ in the convergence rate is linked to the fact that, for a given $\beta$, two solutions $w_1$ and $w_2$ corresponding to different, compactly supported initial conditions, satisfy, for a certain constant $K$ (see Subsection \ref{ssec_2sol}):
\[
\| w(\tau, \cdot) - w_1 (\tau, \cdot) \|_1 \leq K {\rm e}^{-\mu \tau}.
\]
\end{rem}

Figure \ref{fig_f_fit} presents as examples three cases that exhibit a certain diversity : $\mu = 0,9$, $\mu = 0,5$ and $\mu = 0,2$. We plot in red dots the evolution along $\tau$ of the simulated value of $\| w (\tau, \cdot) - W(\tau, \cdot) \|_1$ and use  function $f$ defined in equation \eqref{eqn_f} to fit the results (blue curves). The fit results are given in the figures, $\pm$ one standard deviation.

We first note that in the three panels of figure \ref{fig_f_fit}, $C \approx 0.1$ as expected and our estimates for $\lambda$ are very close to $\mu$. Note that in 
the second panel, with $\mu=0.5$, the values of $A$ and $B$ cannot be estimated independently thus the large inaccuracy/variance on their determination. Finally, 
the third panel shows a marked discontinuity around $\tau=6$. This is due to the discretisation of the age distribution: with small values of $\mu$, the number of random walkers that have never experienced a single renewal during the simulation period becomes large. Since, according to our initial conditions, all walkers have the same initial age, many walkers will enter the last age bin simultaneously thus causing the observed discontinuity. However even in this case, we obtain a very good fit for $\lambda$ by restricting the fit to the values before the discontinuity and fixing $C$ to 0.8. 

Figure \ref{fig_exp_fits_results} summarizes the values of $\lambda$ determined from Monte-Carlo simulations identical to those shown in Fig.\ref{fig_f_fit} (red crosses), together with the diagonal line $\lambda = \mu$ (blue). For all the values of $\mu$ tested, the simulations confirm that $w$ tends to $W$ with a sum of exponential rates given by $\mu$ and $1-\mu$. Therefore, taken together, those simulation results, while agreeing with our analytical estimations, suggest that our estimate of $\| w - W \|_1$ may not be optimal, in particular for larger values of $\mu$.

\begin{figure}[H]
\begin{minipage}{\textwidth}
\begin{minipage}{.7\textwidth}
\centering{\includegraphics[width=.9\linewidth]{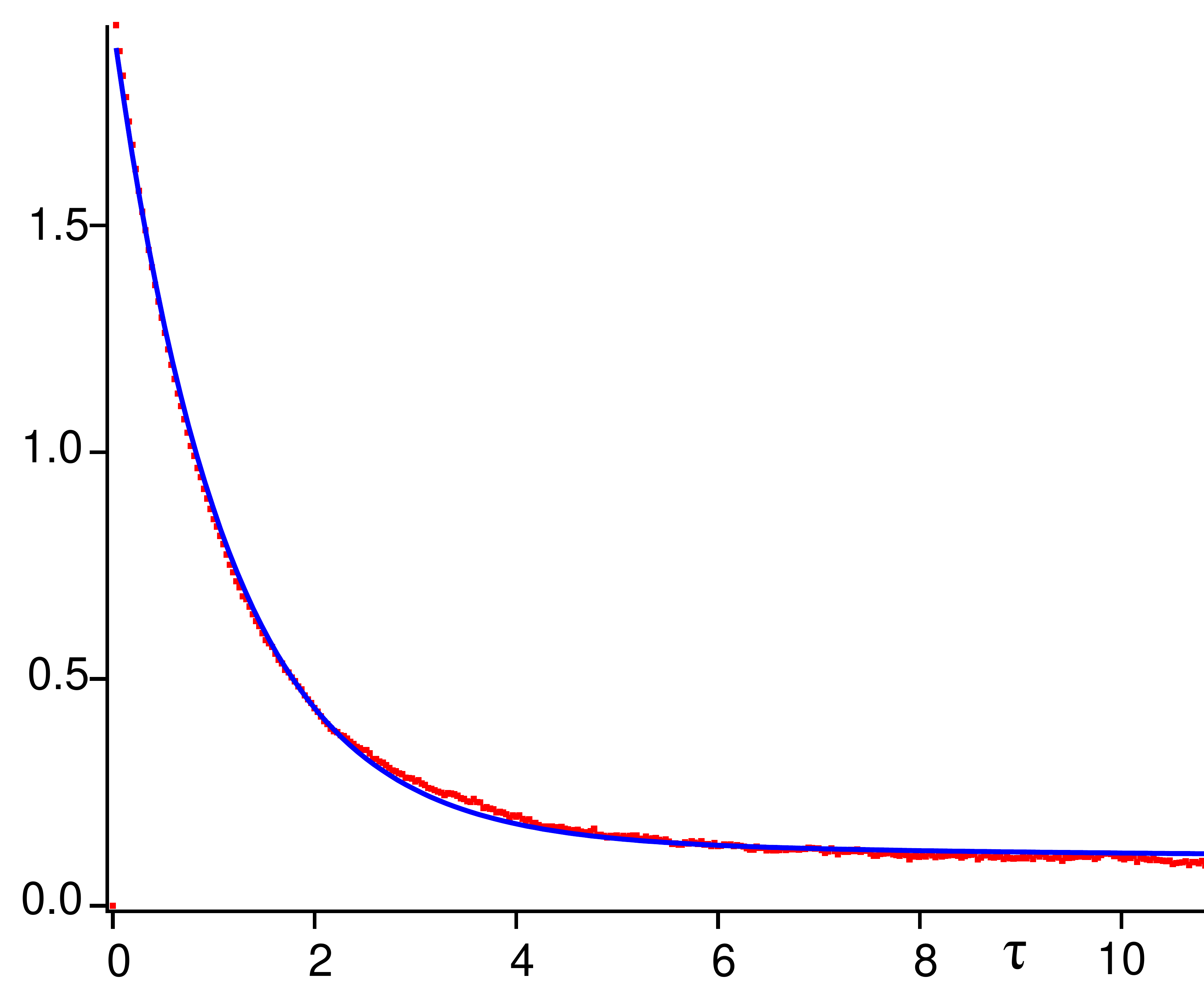}} 
\end{minipage}%
\begin{minipage}{.35\textwidth}
{$\mu = 0.9$. \medskip\\ 
\label{fig_f_fit_a}
$\lambda = 0.899 \pm 0.007$ \\
$ A = 1.795 \pm 0.007 $ \\
$ B = 0.050 \pm 0.008 $ \\
$ C = 0.098 \pm 0.004 $ }
\end{minipage}
\end{minipage}

\begin{minipage}{\textwidth}
\begin{minipage}{.7\textwidth}
\centering{\includegraphics[width=.9\linewidth]{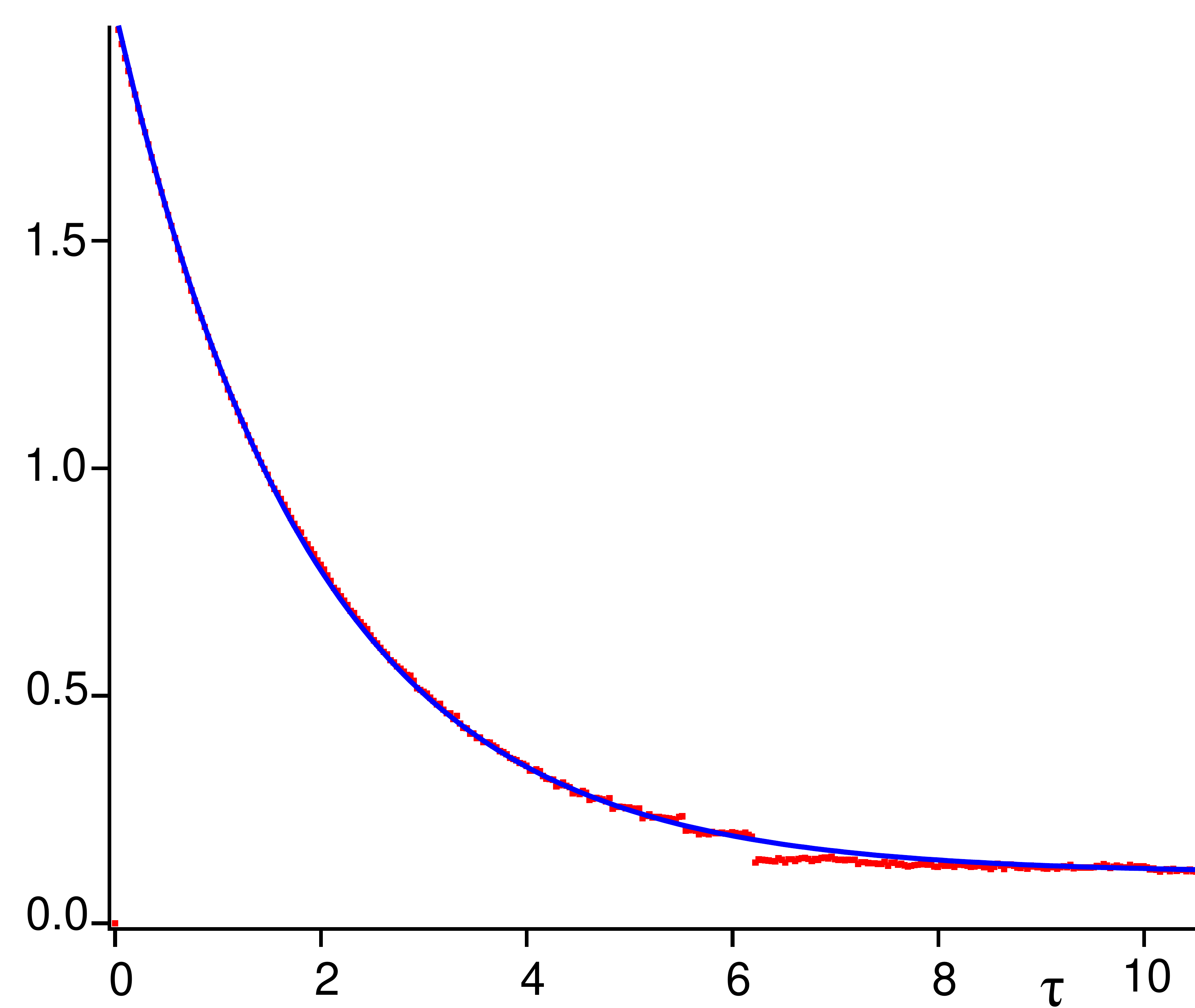}} 
\end{minipage}%
\begin{minipage}{.35\textwidth}
{$\mu = 0.5$. \medskip\\ 
\label{fig_f_fit_b}
$\lambda = 0.523 \pm 0.033$ \\
$ A = 1.895 \pm 1.240 $ \\
$ B = -7.560 \times 10^{-5} \pm 1.240 $ \\
$ C = 0.110 \pm 0.001 $.}
\end{minipage}
\end{minipage}

\begin{minipage}{\textwidth}
\begin{minipage}{.7\textwidth}
\centering{\includegraphics[width=.9\linewidth]{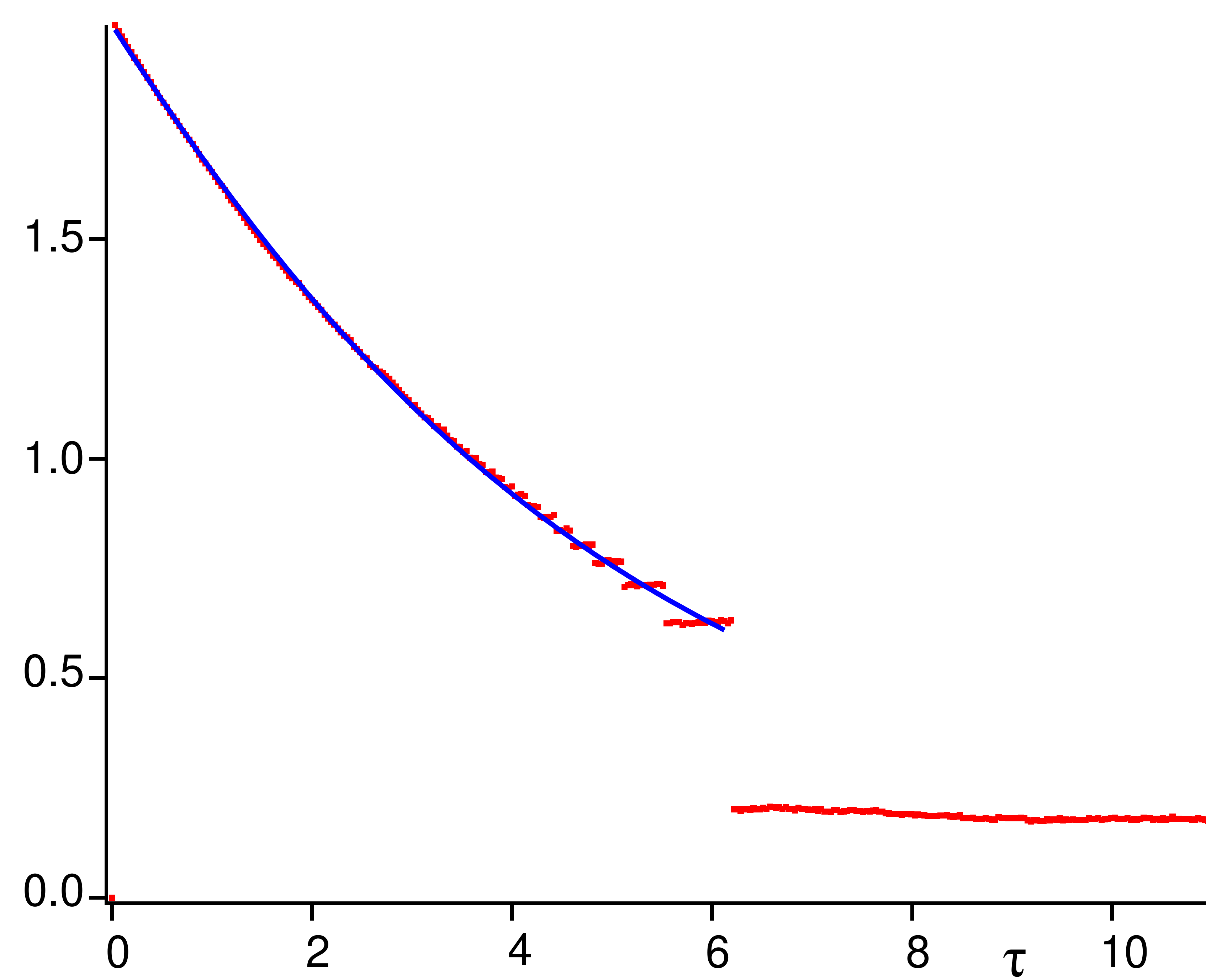}} 
\end{minipage}%
\begin{minipage}{.35\textwidth}
{$\mu = 0.2$. \medskip\\
\label{fig_f_fit_c}
$\lambda = 0.219 \pm 0.001$ \\
$ A = 2.030 \pm 0.013 $ \\
$ B = -0.121 \pm 0.015 $ \\
$ C = 0.08 \pm 0 $}
\end{minipage}
\end{minipage}

\caption{Fit by $f$ defined in equation \eqref{eqn_f}  (blue curves), for different $\mu$, of the simulated $\| w(\tau, \cdot) - W(\tau, \cdot) \|_1$ (red dots). Initial age distribution: $w^0 = \delta_0$.}
\label{fig_f_fit}
\end{figure}

\begin{figure}[H]
\centering{\includegraphics[width=\linewidth]{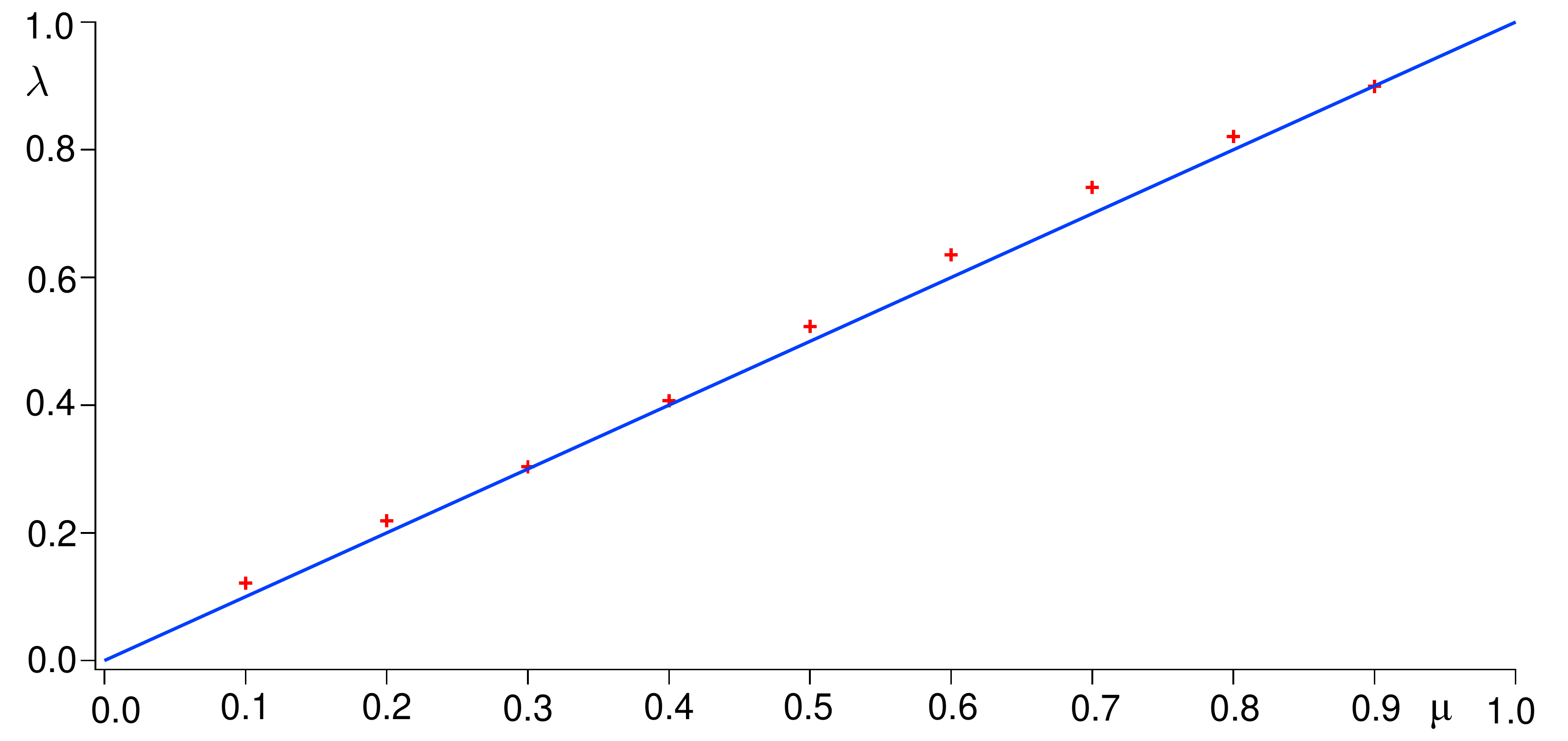}} 
\caption{Values of the exponent $\lambda$ found by using function $f$ from equation \eqref{eqn_f} to fit the simulated values of  $\| w(\tau, \cdot) - W(\tau, \cdot) \|_1$ for $\mu \in \{ 0,1 \,;\; 0,2 \,;\; \dots \,;\; 0,9\}$.}
\label{fig_exp_fits_results}
\end{figure}

\subsection{For large $\mu$, $W$ provides a better asymptotic approximation of $w$ than $W_\infty$}

Figure~\ref{fig_mu_slopes} compares the distances between $w$ and $W$ (red dots), $w$ and $W_\infty$ (black dots), or $W$ and  $W_\infty$ (green curve), for three values of $\mu$. This figure shows that for for $\mu \leq 0.5$, $W(\tau, \cdot)$ and $W_\infty$ are systematically much closer to each other than to $w$. However, as $\mu$ increases, this trend reverses: for large enough $\tau$, $w$ becomes significantly closer to $W$ than to $W_\infty$: the distance between $w$ and $W$ converges much faster. Therefore, according to those simulation results $W$ is a much better asymptotic approximation $w$ for $\mu>0.5$, thus justifying further its utility here.

\begin{figure}
\centering
\begin{minipage}{\textwidth}
  \centering
  \includegraphics[width=0.8\linewidth]{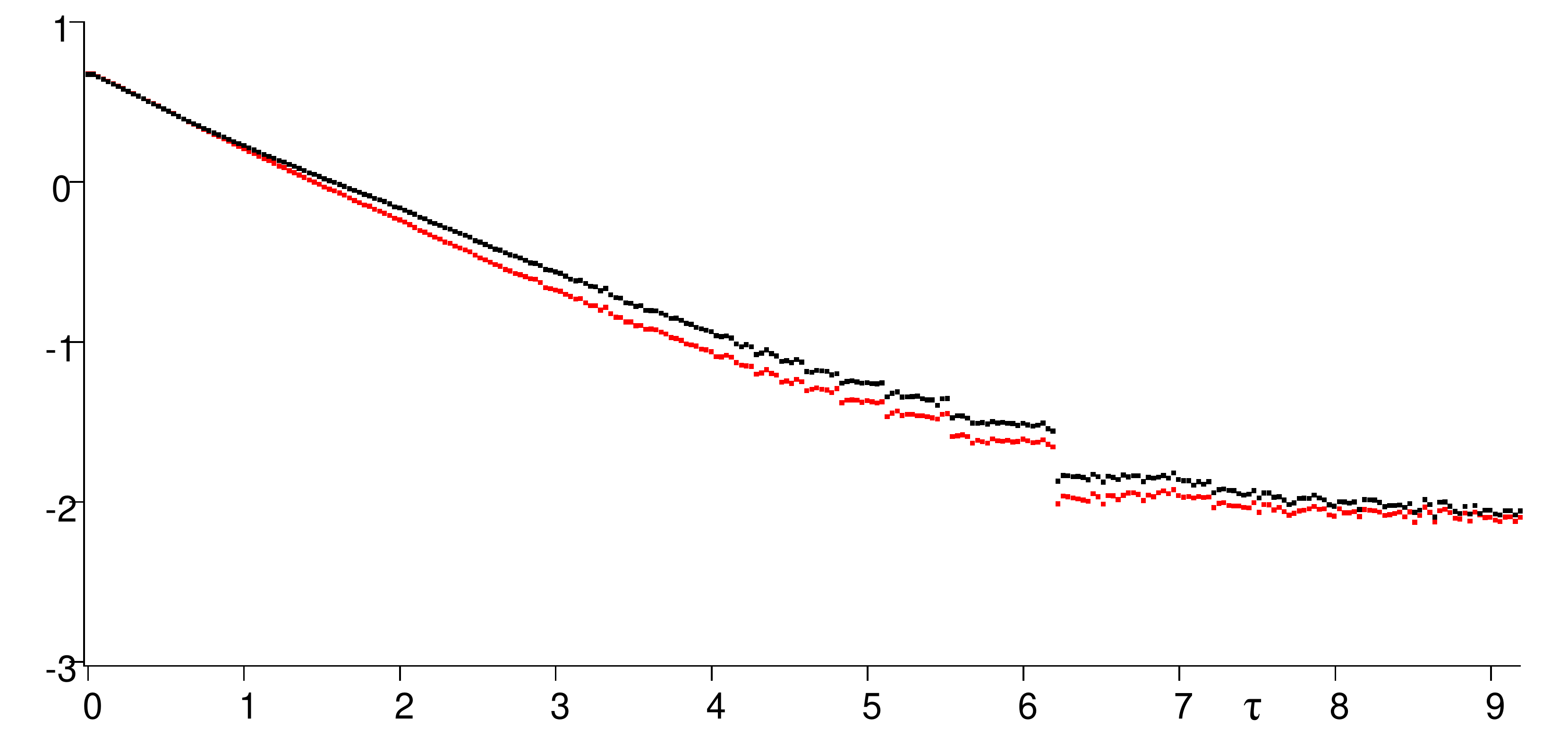} 

  $\mu = 0.5$
  \label{fig_sub_05}
\end{minipage}
\begin{minipage}{\textwidth}
  \centering
  \includegraphics[width=0.8\linewidth]{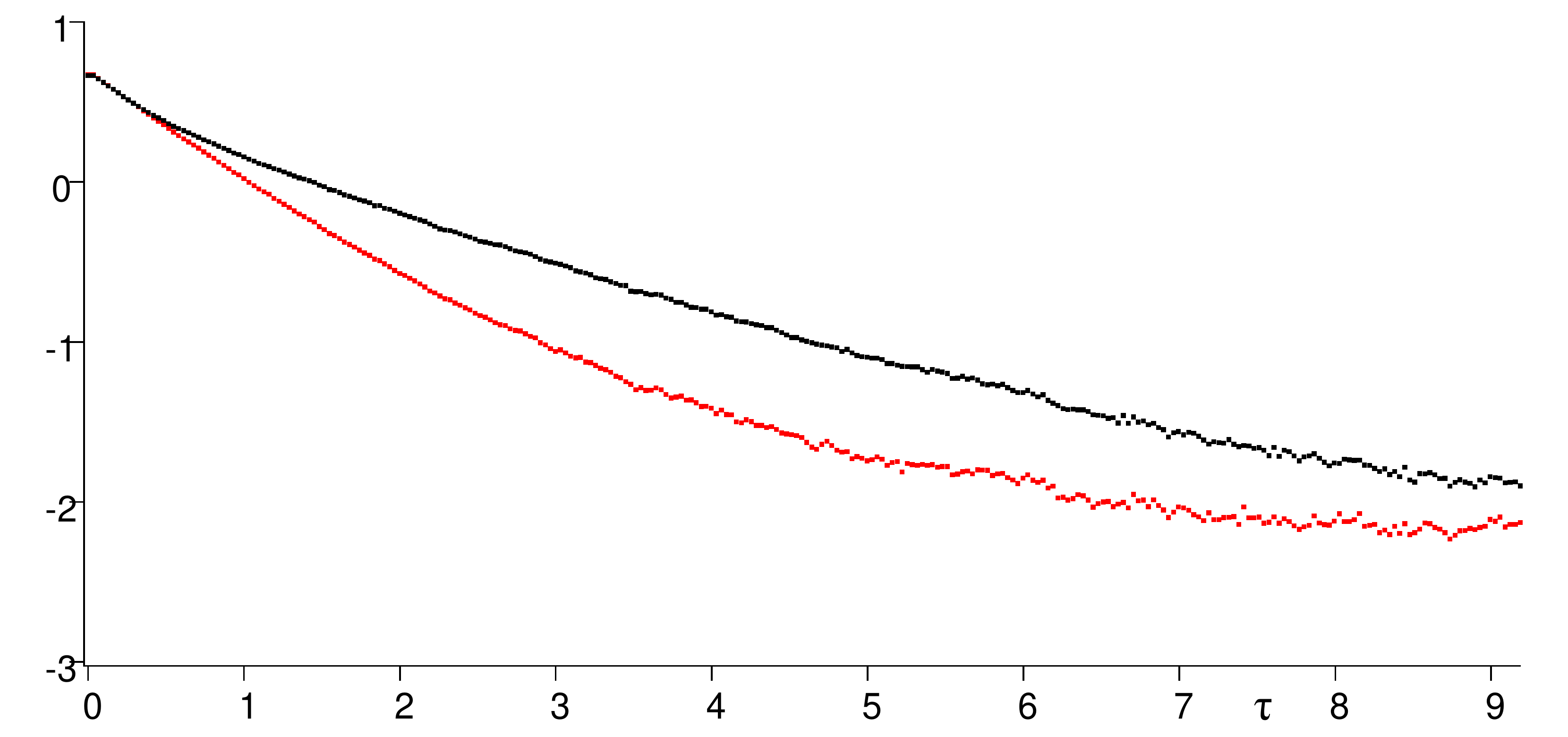} 

  $\mu = 0.7$
  \label{fig_sub_07}
\end{minipage}
\begin{minipage}{\textwidth}
  \centering
  \includegraphics[width=0.8\linewidth]{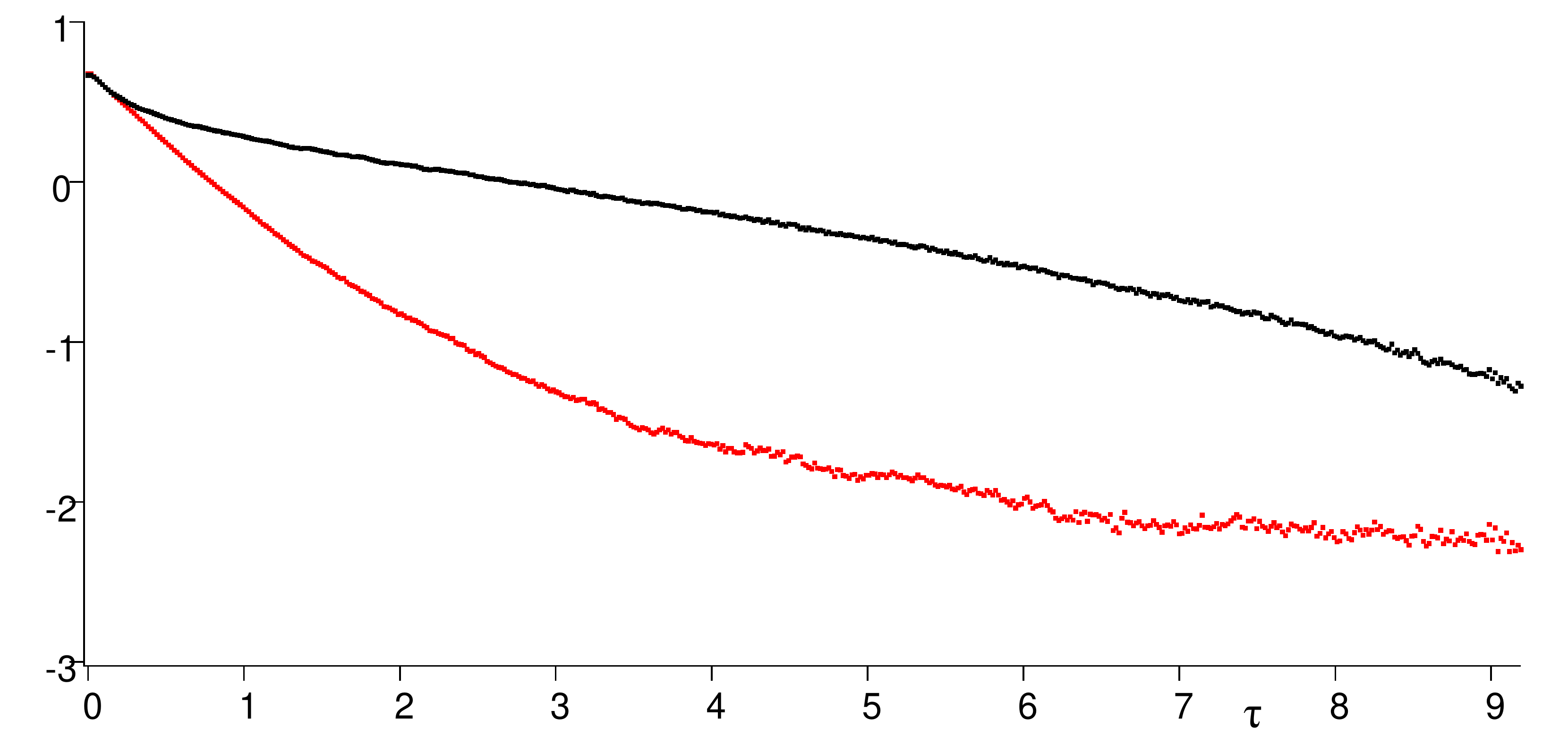} 

$\mu = 0.9$
  \label{fig_sub_09}
\end{minipage}
\caption{Influence of $\mu$ on $\ln \| w(\tau, \cdot) - W(\tau,\cdot) \|_1$ (red dots) and $\ln \| w - W_\infty \|$ (black dots): for higher values of $\mu$, $w$ is significantly closer to $W$ than to $W_\infty$.}
\label{fig_mu_slopes}
\end{figure}

\section{Future developments}

Throughout the article we have estimated $L^1$ norms, but we have presented the estimates in the context of an entropic structure. It is indeed possible by means analogous to ours to prove entropy inequalities for dissipations corresponding to other $H$ functions than $| 1 - \cdot |$. For instance, the classical $H(x) = x \ln x - x + 1$ also allows us to prove a convergence rate of the corresponding entropy to $0$: it is also $K ( {\rm e}^{-\mu \tau} + {\rm e}^{(\mu - 1) \tau} )$. Thanks to the Csisz\'ar-Kullback inequality, it is also possible to prove a rate of convergence of $\| w(\tau, \cdot) - W(\tau, \cdot) \|_1$ to $0$, albeit one worse than that obtained in theorems~\ref{thm_beta_ref} and \ref{thm_beta_gral}.

We may encounter inequalities such as that of proposition~\ref{prop_H'}, bounding the derivative of an entropy with respect to a probability measure $W {\rm d} b$ by an entropy dissipation with respect to another measure (which we can compare to the dissipation with respect to a probability measure ${\rm d} \gamma_\tau$). When the comparison of $DH(u|W {\rm d} b)$ and $DH(u | {\rm d} \gamma )$ does not follow calculations as straightforward as ours, an alternative may be to rely on a precise Jensen estimate comparing the entropy dissipations with respect to two absolutely continuous probability measures.

Here, we have considered a spatially-homogeneous (zero-dimensional), age-dependent renewal probability $\beta(a)$. We believe the ideas we have exposed may be used to tackle the problem with a spatial extension, for instance in a discrete space setting.

One major interest of our age-structure approach of CTRW is that the dynamics remain Markovian. We believe that keeping Markovian properties will be crucially helpful when introducing the coupling between sub-diffusive CTRW and reaction, since the coupling should simply consist in the addition of the reaction and the subdiffusion terms (contrarily to the case of fractional dynamics). However, the extent to which the supplementary age variable will make this process more complex remains to be evaluated.

\section{Appendix}

\subsection*{The case $\mu=1$}

It is quite interesting to notice that even if the behaviour is not really self similar, our method gives a precise asymptotic for the case $\mu=1$. To illustrate this, we focus on the reference case: $\beta(a)=\frac{1}{1+a}$. In this case the 'pseudo equilibrium reads'
$$
W(\tau,b)=\frac{1}{({\rm e}^{-\tau}+b)\log(1+{\rm e}^{\tau})}.
$$
This pseudo equilibrium tends to a Dirac mass but gives still quantitative information. Indeed, following the same computation than for equation \eqref{eq_DHL1} for the case $\mu<1$, we obtain easily
$$
\frac{d}{d\tau}\int_0^1 |w-W|\leq -\frac{{\rm e}^\tau}{1+{\rm e}^\tau}\int_0^1 |w-W|+2\left|C(\tau)\delta(\tau)\right|.
$$
Where we also have 
$$
C(\tau)\delta(\tau) =\int_0^1 \left(\frac{{\rm e}^\tau b}{1+{\rm e}^\tau b}-1\right)W(\tau,b)db.
$$
This leads to 
$$
C(\tau)\delta(\tau)=-\int_0^1 \frac{{\rm e}^{-\tau}}{({\rm e}^{-\tau}+b)^2\log(1+{\rm e}^\tau)}=\frac{{\rm e}^{-\tau}}{\log(1+{\rm e}^\tau)}\left(\frac{1}{{\rm e}^{-\tau}+1}-\frac{1}{{\rm e}^{-\tau}}\right).
$$
And finally, 
$$
C(\tau)\delta(\tau)=-\frac{{\rm e}^{\tau}}{(1+{\rm e}^{\tau})\log(1+{\rm e}^{\tau})}\rightarrow 0.
$$
And we can still claim that $\int_0^1 |w-W|\rightarrow 0$. We can give a (rough) estimate for a rate of convergence. Integrating, we have 
$$
\int_0^1 |w-W|\leq \frac{1}{1+{\rm e}^\tau}\int_0^1 |w-W|(\tau=0)+ 2 \frac{1}{1+{\rm e}^\tau}\int_0^\tau \frac{{\rm e}^{\tau'}}{\log(1+{\rm e}^{\tau'})}d\tau' .
$$
We estimate the second term
$$
 \frac{1}{1+{\rm e}^\tau}\int_0^\tau \frac{{\rm e}^{\tau'}}{\log(1+{\rm e}^{\tau'})}d\tau'= \frac{1}{1+{\rm e}^\tau}\int_1^{{\rm e}^\tau} \frac{1}{\log(1+u)}du
 $$
This term behaves as $\tau$. Indeed, we have easily (splitting the integral at ${\rm e}^{\alpha\tau} $ for $\alpha<1$.

 $$
 \frac{1}{\log(1+{\rm e}^\tau)}\leq \frac{1}{1+{\rm e}^\tau}\int_1^{{\rm e}^\tau} \frac{1}{\log(1+u)}du\leq \frac{{\rm e}^{(\alpha-1) \tau}}{1+{\rm e}^{-\tau}}+\frac{1}{\log(1+{\rm e}^{\alpha\tau})}
 $$
 $$
\int_0^1 |w-W|\leq \frac{1}{1+{\rm e}^\tau}\int_0^1 |w-W|(\tau=0)+ \frac{K}{1+\tau}\leq\frac{K'}{1+\tau}  .
 $$


\section*{Acknowledgements}
This work was initiated within the framework of the LABEX MILYON (ANR-10-LABX-0070) of Universit\'e de Lyon, within the program "Investissements d'Avenir'' (ANR-11-IDEX-0007) operated by the French National Research Agency (ANR). \\

We wish to thank Sergei Fedotov for many valuable discussions. This work could not have been written without the help of Vincent Calvez.

\pagebreak
\bibliographystyle{plain} 
\bibliography{bib-Berry-Lepoutre-Mateos} 
\end{document}